\documentclass[11pt]{amsart}
\usepackage[utf8]{inputenc}
\usepackage[T1]{fontenc}
\usepackage[english]{babel}
\usepackage{graphicx}
\usepackage{amsmath}
\usepackage{amssymb}
\usepackage{amsthm}
\usepackage{fullpage}
\usepackage{mathrsfs}
\usepackage{mathtools}
\usepackage{tikz-cd}
\usepackage[hidelinks]{hyperref}
\usepackage{cleveref}

\hypersetup{allcolors=black}

\newtheorem{theorem}{Theorem}
\numberwithin{theorem}{section}
\newtheorem{lemma}[theorem]{Lemma}
\newtheorem{proposition}[theorem]{Proposition}

\newtheorem{remark}[theorem]{Remark}
\newtheorem{corollary}[theorem]{Corollary}
\newtheorem{question}{Question}
\newtheorem{definition}[theorem]{Definition}

\newtheorem*{thm*}{Theorem A}

\newcommand\N{\mathbb N}

\renewcommand\AA{\mathscr A}
\newcommand\BB{\mathscr B}
\newcommand\CC{\mathscr C}
\newcommand\PP{\mathscr P}
\newcommand\Bb{\mathcal B}

\newcommand\Ff{\mathcal F}
\newcommand\Mm{\mathcal M}
\newcommand\Pp{\mathcal P}
\newcommand\Tt{\mathcal T}
\newcommand{\bZ}{\overline Z}
\newcommand\qi{\simeq_q}
\DeclareMathOperator\Adj{Adj}

\DeclareMathOperator\Diag{Diag}
\DeclarePairedDelimiter{\abs}{\lvert}{\rvert}
\DeclarePairedDelimiter{\norm}{\lVert}{\rVert}
\makeatletter
\let\oldabs\abs
\def\abs{\@ifstar{\oldabs}{\oldabs*}}
\let\oldnorm\norm
\def\norm{\@ifstar{\oldnorm}{\oldnorm*}}
\let\oldbracket\bracket
\def\bracket{\@ifstar{\oldbracket}{\oldbracket*}}
\let\oldpare\pare
\def\pare{\@ifstar{\oldpare}{\oldpare*}}
\makeatother
\title{{\large Block structures of graphs and quantum isomorphism}}

\author{Amaury Freslon$ ^{1} $, Paul Meunier$ ^{2} $, Pegah Pournajafi$ ^{3} $}
\thanks{\noindent $ ^{1} $Université Paris-Saclay, CNRS, Laboratoire de Mathématiques d’Orsay, 91405 Orsay, France. \\ \indent $ ^{2}$ Department of Mathematics: Analysis, Logic and Discrete Mathematics, Ghent University, Krijgslaan 281, 9000 Gent, Belgium (formerly KU Leuven, Department of Mathematics, Celestijnenlaan 200B box 2400, BE-3001 Leuven), research supported by the grant 11PAL24N funded by the Research Foundation Flanders (FWO). \\ \indent $ ^{3} $Chaire Combinatoire, Collège de France, Université PSL, 75005, Paris, France}

\begin{document}

	\maketitle

	\vspace*{-.3cm}
	
	\begin{abstract}
		We prove that for every pair of quantum isomorphic graphs, their block trees and their block graphs are isomorphic, and that such an isomorphism can be chosen so that the corresponding blocks are quantum isomorphic -- in particular, 2-connectedness is preserved under quantum isomorphism. 
		We conclude with some corollaries, including obtaining some necessary conditions on a pair of quantum isomorphic, not isomorphic graphs with a minimal number of vertices.
	\end{abstract}

	\section*{Introduction}\label{sec:introduction}
	
	In recent years, the notion of quantum isomorphism of graphs gained a major place in the study of noncommutative properties of finite graphs. 
	Indeed, this seemingly candid definition turned out to be related to deep mathematical questions, such as the study of $\Ff$-isomorphism of graphs~\cite{Mancinska2019}, the representation of finitely presented groups~\cite{Mancinska2019QiNotI, Slofstra2020}, or the study of quantum automorphism groups of graphs~\cite{Lupini2020}. 
	While striking results allow for a global description of quantum isomorphism (see~\cite{Mancinska2019} for characterisations related to homomorphism count or representations of compact quantum groups), very little is explicitly known about it. 
	For instance, we still do not know a smallest pair of quantum isomorphic graphs which are not isomorphic: the smallest examples known have 24 vertices~\cite[Theorem~6.5]{Mancinska2019QiNotI}, and it follows from~\cite[Theorem B.2]{Meunier2023} that such graphs need at least 6 vertices. 
	
	To approach such problems, as well as for explicit computations of quantum automorphism groups of graphs, it is of great importance to understand how quantum isomorphism behaves with respect to the structural properties of graphs. 
	In this regard, the behaviour of quantum isomorphism with respect to the decomposition of a graph into its connected components, which was partially used in several works~\cite{Lupini2020, vanDobbenetal2023}, was only recently established in full generality~\cite{Meunier2023}, where it is shown that if two graphs are quantum isomorphic, then their connected components are in bijection and two-by-two quantum isomorphic themselves.

	It is then natural to wonder how quantum isomorphism behaves with respect to the decomposition of a connected graph into its 2-connected components, namely, its blocks. 
	We first establish that 2-connectedness, like connectedness, is preserved under quantum isomorphism.
	Our main result (\Cref{thm:pegah_conjecture}) then describes precisely how quantum isomorphism behaves with respect to two canonical block decompositions of a graph $G$: its block tree $\Tt(G)$ and its block graph $\Bb(G)$. 
	Here, the colours of the block tree refer to its natural bicolouring into blocks of $G$ and cut vertices of $G$.

	\begin{thm*}\label{intro_thm:pegah_conjecture}
		If $G$ and $H$ are two quantum isomorphic connected graphs, then
		\begin{enumerate}
			\item there exists a graph morphism $\alpha: \Tt(G) \rightarrow \Tt(H)$ such that
			\begin{itemize}
				\item $\alpha$ is an isomorphism between the block trees of $G$ and $H$ that preserves the colours,
				\item for every block $b$ of $G$, we have $b \qi \alpha(b)$,
			\end{itemize}
			\item there exists a graph morphism $\beta: \Bb(G) \rightarrow \Bb(H)$ such that 
			\begin{itemize}
				\item $\beta$ is an isomorphism between the block graphs of $G$ and $H$,
				\item for every block $b$ of $G$, we have $b \qi \beta(b)$.
			\end{itemize}
		\end{enumerate}
	\end{thm*}
	We also prove a local version of this result (\Cref{thm:degree_cut}). 
	As immediate corollaries of Theorem~A, we obtain that the number of blocks and cut vertices are preserved under quantum isomorphism, and that two quantum isomorphic but not isomorphic graphs with minimal number of vertices are either 2-connected or have two-by-two isomorphic blocks -- and their complements satisfy the same condition as well. 
	This narrows the search for a minimal pair of quantum isomorphic and not isomorphic graphs. 
	Finally, our main theorem is also key to the systematic study of quantum properties of block graphs and the computation of their quantum automorphism groups carried out in~\cite{FMP20250hyperbolic}.

	Our approach consists in finding a natural way to split a connected graph into a disjoint union of connected graphs with strictly less blocks than the original graph, which allows for inductive arguments. 
	Formally, this splitting operation is encoded in an operation $\Gamma$, detailed in~\Cref{sec:anchored_graphs}. 
	The key property of $\Gamma$ is that it preserves quantum isomorphisms, and it does so in an explicit way: given a quantum isomorphism $U$ from a graph $G$ to a graph $H$, one has an explicit quantum isomorphism $\Gamma(U)$ from $\Gamma(G)$ to $\Gamma(H)$. 
	This allows for an efficient treatment of the proof of our main theorem as well as its local version.

	Let us conclude this introduction by outlining the contents of this work. In~\Cref{sec:preliminaries}, we recall the necessary definitions and introduce our notations. 
	In~\Cref{sec:walks_2connectedness}, we show that 2-connectedness is preserved under quantum isomorphism (\Cref{coro:qi_2connected}). 
	In~\Cref{sec:partitioned_graphs}, we generalise a standard folklore technique by studying quantum isomorphisms preserving given partitions of the vertex sets of the graphs. 
	This technique was only used so far for the partitions into the vertex sets of the connected components, and we show how to use it in full generality in~\Cref{thm:partitioned_graphs}. 
	In~\Cref{sec:anchored_graphs}, we introduce the operation $\Gamma$ and prove that it preserves quantum isomorphisms in~\Cref{lem:qi_for_anchored_graphs}. 
	Finally, in~\Cref{sec:pegah_conjecture}, we obtain our main results.

	{\small
		\setcounter{tocdepth}{1}
		\tableofcontents
	}

	\section{Preliminaries}\label{sec:preliminaries}
	
	Let us start by introducing some notations and providing preliminary results. 
	
	\subsection*{Graph theoretical notions}
	
	For any graph theoretical notion not defined here, we refer to~\cite{BondyMurty}. 
	
	All graphs in this paper are finite and simple -- that is, loopless and without multiple edges. 
	We denote the vertex set and the edge set of a graph $ G $ respectively by $ V(G) $ and $ E(G) $, and its adjacency matrix by $\Adj(G)$. 
	We may write $ xy $ for an element $ \{x,y\} \in E(G) $. 
	For $ x \in V(G) $ the \emph{neighbourhood} of $ x $ is defined as $ N_G(x) = \{y \in V(G) \mid xy \in E(G) \} $ and the \emph{closed neighbourhood} of $ x $ is defined as $ N_G[x] = N_G(x) \cup \{x\} $. 
	
	For $k\geq 1$, a \emph{walk} in $G$ is a sequence $ x_1, \dots, x_k $ of vertices of $G$ such that $x_ix_{i+1} \in E(G)$ or $x_i = x_{i+1}$ for $ 1 \leq i \leq k-1 $, and its \emph{length} is defined to be $k-1$. We denote by $w_i(x,y)$ the number of walks of length $i \geq 0$ from $x$ to $y$ in $G$. It is straightforward to show that $ w_i(x,y)= [\Adj(G)^i]_{xy}$. 
	
	A \emph{graph morphism} (or simply a \emph{morphism}) from a graph $ G $ to a graph $ H $ is a function $ \phi: V(G) \to V(H) $ such that for all $ xy \in E(G) $ we have $ \phi(x)\phi(y) \in E(H) $. 
	The morphism~$ \phi $ is an isomorphism if and only if it is a bijection and for every $ x, y \in V(G) $ such that $ xy\notin E(G) $ we have $ \phi(x)\phi(y) \notin E(H) $. 
	
	We use the notation $ G[S] $ for the subgraph $ H $ of $ G $ induced by a set $ S \subseteq V(G) $, that is, $ V(H) = S $ and $ E(H) = \{ xy \in E(G) \mid x,y \in S \} $. Moreover, we write $G \oplus H$ for the disjoint union of $G$ and $H$, that is $V(G \oplus H) = V(G) \oplus V(H)$ and $E(G \oplus H) = E(G) \oplus E(H)$.

	\subsection*{Connectivity and block structures}
	
	A graph $ G $ is connected if for every $ x,y \in V(G) $, there exists a path from $ x $ to $ y $ in $ G $. 
	The \emph{connected components} of $ G $ are the maximal connected subgraphs of $ G $. 
	If $ G $ has at least two vertices, a \emph{cut vertex} of $ G $ is a vertex $ v $ such that $ G \setminus \{v\} $ has strictly more connected components than $ G $ does. 
	For a graph on one vertex, in this paper, we define its only vertex to be a cut vertex.

	We recall that every graph in this paper is finite, loopless, and without multiple edges. 
	In this setting, a graph $ G $ is \emph{2-connected} if for every $ v \in V(G) $, the graph $ G \setminus \{v\} $ is connected and nonempty. 
	A maximal 2-connected subgraph of $ G $ is called a \emph{block} of $ G $. By a slight abuse, we also use the term block to sometimes refer to the vertex set of said subgraph.
	We call a vertex $ v $ in a block $ b $ of $ G $ an \emph{internal vertex} if it is not contained in any block of $ G $ but $ b $. 
	We use the following properties of blocks of $ G $ several times without necessarily referring to the following proposition. 
	For the proof, see Proposition~5.3 in~\cite{BondyMurty} and the paragraph under its proof.
	
	\begin{proposition} \label{prop:properties_of_blocks}
		For every graph $ G $: 
		\begin{enumerate}
			\item every two blocks of $ G $ are either disjoint or have exactly one vertex in common,
			\item the blocks of $ G $ partition the edge-set of $ G $, 
			\item every cycle of $ G $ is contained in a single block of $ G $,
			\item a nonisolated vertex is an internal vertex of a block of $ G $ if and only if it is not a cut vertex of $ G $. 
		\end{enumerate}
	\end{proposition}
	
	The blocks of a graph form a tree-like structure. Indeed, this is captured by two canonical block structures of a graph $G$, namely the block tree and the block graph of $ G $. 
	
	Let $ G $ be a graph. We denote the set of its blocks by $ \BB(G) $ and the set of its cut vertices by~$\CC(G)$.
	
	First, we define a graph $ \Tt(G) $ (or $ \Tt $ if the underlying graph is clear from the context) by setting $ V(\Tt) = \mathscr B(G) \sqcup \mathscr C(G) $ and $ E(\Tt) = \{ \{c,b \} \mid b \in \mathscr{B}(G), c \in \mathscr{C}(G), c \in b \} $. It follows directly from~\Cref{prop:properties_of_blocks} that $ \Tt(G) $ is a tree and it is called the \emph{block tree of $G$} (see Section 5.2 of~\cite{BondyMurty} for more). 
	Notice that $ \Tt $ always has a natural bipartition (or equivalently, a proper 2-colouring) into vertices that are blocks of $ G $ and vertices that are cut vertices of $ G $. 
	Thus we will sometimes refer to the colour of the former as white and of the latter as black.

	Second, we define a graph $ \Bb(G) $ (or $ \Bb $ if the underlying graph is clear from the context) by setting $ V(\Bb) = \mathscr B(G) $ and $ E(\Bb) = \{ \{b,b'\} \mid b \neq b', b \cap b' \neq \varnothing  \} $, that is, $ \Bb(G) $ is the intersection graph of the blocks of $ G $. 
	It is called the \emph{block graph of $ G $}.
	
	A graph $ H $ is called a \emph{block graph} if there exists a graph $ G $ such that $ H $ is isomorphic to the block graph of $ G $. There are numerous equivalent characterisations of block graphs. In this paper, we use the following characterisation, from~\cite[Theorem A]{Harary1963blockgraphs}.
	
	\begin{proposition} \label{prop:characterisation_blockgraph}
		A graph $ G $ is a block graph if and only if every block of $ G $ is a complete graph.
	\end{proposition}

	\subsection*{Center of graphs}\label{subsec:center_rho}
	
	Every graph $ G $ is equipped with a distance $ d = d_G: V(G) \times V(G) \to~[0, +\infty] $ called the \emph{graph distance} where for $ x, y \in V(G) $, the distance $ d(x,y) $ is the length of a shortest path from $x$ to $y$ in $ G $. 
	Notice that $ G $ is connected if and only if $d_G(\cdot , \cdot) < +\infty $. 
	For every vertex $ x \in V(G) $, the \emph{eccentricity} of $ x $, denoted by $ e(x) $, is defined to be its distance to the vertex furthest from it, that is $e(x) = \max_{y \in V(G)} d(x,y) $. 
	The \emph{center} of $ G $, denoted by $Z(G)$, is the set of vertices with minimum eccentricity, that is $ Z(G) = \{x \in V(G) \mid \forall y \in V(G) \ e(y) \geq e(x) \} $.
	We recall the following theorem.
	
	\begin{theorem}\label{thm:center_block}
		Let $G$ be a connected graph on at least two vertices. 
		There exists a block $b$ of $G$ such that $Z(G) \subseteq b$. 
		Moreover, the center $Z(G)$ is either a singleton consisting in a cut vertex or there exists a unique block containing it.
	\end{theorem}
	
	\begin{proof}
		For the proof of the fact that the center is contained in a block, see Theorem 2.2 in~\cite{BuckleyHarary1990Distancesingraphs}. 
		If $ Z(G) $ contains an internal vertex $ u $, then Item (4) of~\Cref{prop:properties_of_blocks} implies that there exists a unique block $ b $ (namely the unique block containing $ u $) such that $ Z(G) \subseteq b $. 
		If $ Z(G) $ has no internal vertices, then either it contains exactly one cut vertex, in which case there is nothing to prove, or it contains at least two distinct cut vertices $ v $ and $ v' $. 
		In this case, let $ b $ and $ b' $ be blocks containing $ Z(G) $, then $ v$ and $ v'$ are both in $ b \cap b' $, which, by Item (1) of~\Cref{prop:properties_of_blocks} is only possible if $ b = b' $, which completes the proof. 
	\end{proof}	
	
	For a connected graph $G$, we define $\bZ(G)$ as follows: if $Z(G)$ consists of a unique cut vertex, then we define $\bZ(G)=Z(G)$. Otherwise, by~\Cref{thm:center_block}, there exists a unique block $b$ of $G$ containing $Z(G)$, in which case, we define $\bZ(G) = b$. Note that $\bZ(G)$ is either a block or a cut vertex of $G$ and, as we will see in the context of anchored graphs in~\Cref{sec:anchored_graphs}, $(G, \bZ(G))$ is an anchored graph.

	\subsection*{Magic unitaries and quantum isomorphism}
	
	We now recall the notion of quantum isomorphism of graphs. We use the definition from~\cite{Lupini2020}, where the authors relate it to the notion of quantum automorphism groups of graphs, as introduced by Banica in~\cite{Banica2004qaut}, extending work of Bichon on quantum automorphism groups of graphs~\cite{Bichon1999qaut}.
	
	The notion of quantum isomorphism relies on that of a magic unitary, first appearing in the work of Wang~\cite{Wang1998}. Let $X$ be a unital $C^*$-algebra. 
	A \emph{magic unitary} with coefficients in $X$ is a matrix $U = (u_{ij})_{i\in I, j\in J} $ where $ I $ and $ J $ are finite sets and such that for all $ i \in I $ and all  $ j \in J $:
	\begin{enumerate}
		\item $ u_{ij}^2 = u_{ij} $ and $ u_{ij}^* = u_{ij} $,
		\item $ \sum_{k \in I} u_{kj}  = 1 = \sum_{k \in J} u_{ik} $.
	\end{enumerate}  
	
	Notice that (2) implies that $ |I| = |J| $. 
	Also, it is well-known that projections that form a partition of unity in a $C^*$-algebra are orthogonal, hence the rows and the columns of a magic unitary are orthogonal. 
	This implies that $U$ is a unitary in $\Mm_{I,J}(X)$. 
	
	A \emph{quantum isomorphism} from a graph $G$ to a graph $H$ with the same number of vertices is a magic unitary $U$ indexed by $ V(H) \times V(G) $ such that $U \Adj(G) = \Adj(H) U$. 
	
	Man\v{c}inska and Roberson~\cite{Mancinska2019} obtained a purely combinatorial equivalent definition, which we mention here, even though not used in the rest of this article. 
	\begin{theorem}[Theorem 7.16 in~\cite{Mancinska2019}]
		Two graphs $ G $ and $ H $ are quantum isomorphic if and only if for every planar graph $ P $ the number of morphisms from $P$ to $ G $ is equal to the number of morphisms from $ P $ to $ H $. 
	\end{theorem}

	\section{Walks in graphs and 2-connectedness} \label{sec:walks_2connectedness}
	
	In this section, after proving some preliminary lemmas about the number of walks between given vertices in a graph, we obtain two essential results. 
	Firstly, we relate quantum isomorphism between two graphs to their blocks and cut vertices (\Cref{lem:qi_blocks}). 
	Secondly, using this, we prove that 2-connectedness is preserved under quantum isomorphism (\Cref{coro:qi_2connected}). 
	
	\subsection{Walks in graphs}
	Let $G$ be a graph, let $x$, $y$, $z\in V(G)$, and fix $i\geq 0$.  
	We denote by $w_i(x,z;y)$ the number of walks of length $i$ from $x$ to $z$ going through $y$, and denote by $w_i^s(x,z;y)$ the number of walks of length $i$ from $x$ to $z$ going through $y$ exactly once. 
	In particular, we have that $w_i^s(x,x;x) = 0 $ if $ i \geq 1$ and that $w_i^s(x,x;x) = 1 $ if $ i=0$. 
	Moreover, it is straightforward that for any $x, y, z \in V(G)$, we have that $w_i(x,z;y)=w_i(z,x;y) $ and $ w_i^s(x,z;y)=w_i^s(z,x;y)$.
	
	\begin{lemma}\label{lem:walk_formula}
		For every $i\geq 0$, we have:
		\[ w_i(x,z;y) = \sum_{j+k+l=i} w_j^s(x,y;y)w_k(y,y)w_l^s(y,z;y).\]
	\end{lemma}
	
	\begin{proof}
		Observe that every walk of length $i\geq 0$ from $x$ to $z$ going through $y$ decomposes uniquely as the concatenation of a walk from $x$ to $y$ passing through $y$ only once (at the end), a walk from $y$ to $y$, and a walk from $y$ to $z$ passing through $y$ only once (at the beginning). Denoting the respective lengths of these walks by $j$, $k$, and $l$, we have $j+k+l=i$. 
		Conversely, given three such walks, their concatenation gives a walk of length $i$ from $x$ to $z$ passing through $y$. Note that these correspondences are inverse of each other, hence we obtain a bijection between such triples of walks and the walks of length $i$ from $x$ to $z$ going through $y$. 
		This leads to the desired formula.
	\end{proof}
	
	The following lemma is based on a claim used in the proof of Theorem 6 of~\cite{Kiefer2019WL}. We prove it here for the sake of completeness.
	
	\begin{lemma}\label{lem:nb_walks}
		Let $G$ and $H$ be two graphs and let $x$, $y$, $z\in V(G)$, and $a$, $b$, $c\in V(H)$. 
		Assume that for every $i\geq 0$, we have:
		\begin{enumerate}
			\item $w_i(x,y) = w_i(a,b)$,
			\item $w_i(y,y) = w_i(b,b)$,
			\item $w_i(y,z) = w_i(b,c)$.
		\end{enumerate}
		Then for every $i\geq 0$ we have $w_i(x,z;y) = w_i(a,c;b)$. 
	\end{lemma}
	
	\begin{proof}
		First, note that for every $j < d_G(x,y)$, we have $w_j(a,b) = w_j(x,y) = 0$ and for $ j = d_G(x,y)$, we have $w_j(a,b) = w_j(x,y) >0$. Hence $d_H(a,b) = d_G(x,y)$.
		
		Let us show that for all $i\geq 0$ we have $w_i^s(x,y;y) = w_i^s(a,b;b)$. Note that for $i<d_G(x,y) = d_H(a,b)$, we have $w_i^s(x,y;y) = 0 = w_i^s(a,b;b)$, so the equality holds. We prove the equality for $i \geq d_G(x,y)$ by induction on $i$. 
		
		Let $i=d_G(x,y) =d_H(a,b)$. 
		In this case, any walk of length $i$ from $x$ to $y$ (resp.\ $a$ to $b$) is minimal, so it goes through $y$ (resp.\ $b$) only once. Hence, we have $$w_i^s(x,y;y) = w_i(x,y) = w_i(a,b) = w_i^s(a,b;b),$$ as desired. 
		
		Now consider $ i \geq d_G(x,y)$ and assume that $w_j^s(x,y;y) = w_j^s(a,b;b)$ for any $j \leq i$. Let us show that $w_{i+1}^s(x,y;y) = w_{i+1}^s(a,b;b)$.
		Setting $z=y$ in~\Cref{lem:walk_formula}, we obtain
		$$
		w_{i+1}(x,y) = w_{i+1}(x,y;y) = \sum_{j+k+l = i+1} w_j^s(x,y;y)w_k(y,y)w_l^s(y,y;y) = \sum_{j+k = i+1} w_j^s(x,y;y)w_k(y,y).
		$$
		Similarly, we obtain $w_{i+1}(a,b) = \sum_{j+k = i+1} w_j^s(a,b;b)w_k(b,b)$. 
		Using the induction hypothesis and the assumptions of the statement, we reach
		\begin{align*}
			w_{i+1}^s(x,y;y) &= w_{i+1}^s(x,y;y)w_0(y,y) \\ 
			& = w_{i+1}(x,y)-\sum_{j=0}^i w_j^s(x,y;y)w_{i+1-j}(y,y)\\
			&= w_{i+1}(a,b) -\sum_{j=0}^i w_j^s(a,b;b)w_{i+1-j}(b,b) \\
			& = w_{i+1}^s(a,b;b),
		\end{align*}
		as desired, completing the inductive proof. 
		
		This shows that for all $i\geq 0$ we have $w_i^s(x,y;y) = w_i^s(a,b;b)$. 
		Applying the same argument to $z$, $y$, $c$, and $b$, we obtain that for all $i\geq 0$, we have
		$$
		w_{i}^s(y,z;y) = w_{i}^s(z,y;y) = w_{i}^s(c,b;b) = w_{i}^s(b,c;b).  
		$$

		Finally, let $i\geq 0$. 
		By~\Cref{lem:walk_formula}, we have
		\begin{align*}
			w_i(x,z;y) &= \sum_{j+k+l=i} w_j^s(x,y;y)w_k(y,y)w_l^s(y,z;y)\\
			&= \sum_{j+k+l=i} w_j^s(a,b;b)w_k(b,b)w_l^s(b,c;b)\\
			&= w_i(a,c;b).
		\end{align*}
		This concludes the proof.
	\end{proof}

	\subsection{2-connectedness}
	
	We start by a reformulation of a lemma due to Fulton (see~\cite[Section 3.2]{FultonThesis}) whose proof we include for the sake of completeness.
	
	\begin{lemma}\label{lem:Fulton}
		Let $U$ be a quantum isomorphism from $G$ to $H$. 
		Let $x$, $y\in V(G)$ and $a$, $b\in V(H)$. 
		If $u_{ax}u_{by}\neq 0$, then for all $i\geq 0$ we have $w_i(x,y) = w_i(a,b)$. 
		In particular, we have $d_G(x,y) = d_H(a,b)$.
	\end{lemma}
	
	\begin{proof}
		Let $A = \Adj(G)$ and $B = \Adj(H)$. 
		Let $i\geq 0$. 
		We have $UA^i = B^iU$, so 
		$$\sum_{z\in V(G)} u_{az} [A^i]_{zy} = [UA^i]_{ay} = [B^iU]_{ay} = \sum_{c\in V(H)} [B^i]_{ac} u_{cy}.$$ 
		Multiplying on the left by $u_{ax}$ and on the right by $u_{by}$ and using the orthogonality of the rows and columns of $U$, we obtain $$u_{ax}u_{by} [A^i]_{xy} = [B^i]_{ab} u_{ax}u_{by}.$$ 
		Since $u_{ax}u_{by}\neq 0$, we have that $[A^i]_{xy} = [B^i]_{ab}$. 
		Recalling that $w_i(x,y) = [A^i]_{xy}$ and that $w_i(a,b) = [B^i]_{ab}$, we obtain the desired result.
		
		Let us prove the last assertion by contrapositive. 
		Up to symmetry, let us assume that $d(a,b) > d(x,y) = k\in \N$. 
		Hence, we have $w_k(x,y) > 0 = w_k(a,b)$, so by what precedes we have $u_{ax}u_{yb}=0$, as desired. 
		This concludes the proof.
	\end{proof}
	
	The following theorem relates the coefficients of a quantum isomorphism from a graph $G$ to a graph $H$ to the blocks of $G$ and $H$. 
	
	\begin{theorem}\label{lem:qi_blocks}
		Let $G$ and $H$ be two quantum isomorphic graphs. Let $U$ be a quantum isomorphism from $G$ to $H$. 
		Let $x, y \in V(G)$ and $a, b \in V(H)$. 
		If $x$ and $y$ are in a common block but not $a$ and $b$, or if $a$ and $b$ are in a common block but not $x$ and $y$, then $u_{ax}u_{by} = 0.$ 
		
		Moreover, if $x$ is a cut vertex and $a$ is not, or $a$ is a cut vertex and $x$ is not, then $u_{ax} = 0$. 
	\end{theorem}
	
	\begin{proof}
		Let $a$, $b\in V(H)$ not be in a common block, and take $x$, $y\in V(G)$ such that $u_{ax}u_{by} \neq 0$. 
		We want to show that $x$ and $y$ are not in a common block neither.

		If $a$ and $b$ are not in the same connected component, then by~\Cref{lem:Fulton}, $x$ and $y$ are not in a common connected component neither, in particular, they are not in a common block of $G$. Thus, from now on, we may assume that $a$ and $b$ are in a common connected component. 
		Hence, there is a cut vertex $c \in V(H)$ distinct from $a$ and $b$ such that every walk from $a$ to $b$ passes through~$c$. 
		In other words, we have $w_i(a,b;c) = w_i(a,b)$ for all $i\geq 0$. 
		Now, we have
		\[ 0 \neq u_{ax}1u_{by} = u_{ax}\left(\sum_{z\in V(G)}u_{cz}\right)u_{by} = \sum_{z\in V(G)} u_{ax}u_{cz}u_{by},\]
		hence there exists $z\in V(G)$ such that $u_{ax}u_{cz}u_{by}\neq 0$. 
		In particular, we have $u_{ax}u_{cz}\neq 0$, $u_{cz}\neq 0$, and $u_{cz}u_{by}\neq 0$, so $z$ is different from $x$ and $y$. Moreover, by~\Cref{lem:Fulton} for all $i\geq 0$ we have $w_i(x,z) = w_i(a,c)$, $w_i(z,z) = w_i(c,c)$, and $w_i(z,y) = w_i(c,b)$. 
		Hence, by~\Cref{lem:nb_walks}, we have $w_i(x,y;z) = w_i(a,b;c) = w_i(a,b)$ since $c$ is a cut vertex separating $a$ and $b$. 
		Since by assumption $u_{ax}u_{by}\neq 0$, by~\Cref{lem:Fulton} again we have that $w_i(x,y) = w_i(a,b)$, so we have shown that $w_i(x,y) = w_i(x,y;z)$. 
		This implies that every walk from $x$ to $y$ passes through $z$. Therefore, $x$ and $y$ are in distinct connected components of $G\setminus z$, that is, $z$ is a cut vertex separating $x$ and $y$. Hence, $x$ and $y$ are not in the same block of $G$.
		
		Applying what precedes to $U^*$ and $H$ and $G$, we obtain that if $u_{ax}u_{by} \neq 0$ and $x$ and $y$ are not in a common block, then neither are $a$ and $b$. 
		This concludes the proof of the first part of the lemma.
		
		Now let $x$ be a cut vertex in $G$ and take $a\in V(H)$ such that $u_{ax}\neq 0$. First, if $G$ has one vertex, then $G = K_1=H$. In this case $a$ is a cut vertex of $H$ as well, and there is nothing to prove. So, from now on, we may assume that $G$ has at least two vertices.  
		By definition, there are $y$ and $z\in V(G)$ which are neighbors of $x$ but are not in a common block of $G$. 
		Now, we have:
		\begin{align*}
			0 &\neq u_{ax} = u_{ax}^2\\
			&= u_{ax} \left(\sum_{b\in V(H)} u_{by}\right)\left(\sum_{c\in V(H)} u_{cz}\right) u_{ax}\\
			&= \sum_{b,c\in V(H)} u_{ax}u_{by}u_{cz}u_{ax}.
		\end{align*}
		In particular, there exist $b$, $c\in V(H)$ such that $u_{ax}u_{by}u_{cz}u_{ax}\neq 0$. 
		This implies that $u_{ax}u_{by}\neq 0$ and that $u_{cz}u_{ax}\neq 0$, so by~\Cref{lem:Fulton} we have that $ab\in E(H)$ and $ca\in E(H)$. 
		Moreover, we have that $u_{by}u_{cz}\neq 0$, so by what precedes $b$ and $c$ are not in a common block of $H$, since $y$ and $z$ are not. 
		Hence, $b$ and $c$ are neighbors of $a$ which are not in a common block of $H$. 
		So $a$ is in the intersection of two distinct blocks, thus it is a cut vertex. 
		Applying what we proved to $U^*$, we obtain that if $a$ is a cut vertex and $u_{ax}\neq 0$, then $x$ is a cut vertex. 
		This concludes the proof.
	\end{proof}
	
	This allows us to show that 2-connectedness is preserved under quantum isomorphism.
	
	\begin{theorem}\label{coro:qi_2connected}
		If $G$ is 2-connected and quantum isomorphic to $H$, then $H$ is 2-connected.
	\end{theorem}
	
	\begin{proof}
		Let us prove it by contrapositive. 
		Assume that $H$ is not 2-connected. If $H$ is not connected, then neither is $G$, therefore $G$ is not 2-connected. So, we may assume that $H$ is connected, and therefore there exists a cut vertex $a\in V(H)$. 
		Let $U$ be a quantum isomorphism from $G$ to $H$. 
		Since $1 = \sum_{x\in V(G)} u_{ax}$, there is $x\in V(G)$ such that $u_{ax} \neq 0$. 
		Hence $x$ is a cut vertex by~\Cref{lem:qi_blocks}, so $G$ is not 2-connected. 
		This concludes the proof.
	\end{proof}

	We conclude this section with an open problem arising naturally from~\Cref{coro:qi_2connected}. 
	
	A connected graph $ G $ is said to be \emph{$ k $-connected} if it has at least $ k $ vertices and for every $ S \subseteq V(G) $ of size at most $ k-1 $, the graph $ G \setminus S $ is connected. 
	The \emph{connectivity} of a connected graph $ G $ is the maximum $ k \in \N $ such that $ G $ is $ k $-connected.
	
	\begin{question}\label{q:quantum_iso_not_connectivity}
		Does there exist a pair of quantum isomorphic graphs $ G $ and $ H $ whose connectivities are not the same? 
		If so, what is the maximum $ k$ such that $ k$-connectivity is preserved under quantum isomorphism?
	\end{question}

	By~\Cref{coro:qi_2connected}, we have that $ k \geq 2 $.
	This question is also of particular interest in light of the notion of cospectrality.
	Two graphs are \emph{cospectral} if the spectra of their adjacency matrices are the same, and it is straightforward to see that quantum isomorphic graphs are cospectral. 
	The answer to~\Cref{q:quantum_iso_not_connectivity} would then be an indicator of where quantum isomorphism stands in the range between cospectrality and isomorphism of graphs. Indeed, should the answer be positive, it generalises the following theorem.
	
	\begin{theorem}\label{thm:cospectral_no_same_connectivity}
		There exist cospectral graphs with different connectivities.
	\end{theorem}
	For examples of such pairs, see~\cite[Theorem 2.1]{Haemers2019} where for every $ r \geq 2 $ it is shown that there exist $ r $-regular cospectral graphs such that the connectivity of one is $2r$ and the connectivity of the other is $ r + 1 $.

	\section{Quantum isomorphism of partitioned graphs}\label{sec:partitioned_graphs}
	
	The aim of this section is to prove~\Cref{thm:partitioned_graphs}. This theorem generalises a standard technique, namely, that the sum on the rows and the columns of some submatrices of a quantum isomorphism are constant under good conditions. This allows for inducing local quantum isomorphisms between subgraphs, and typically happens when the submatrices are indexed by the connected components of the underlying graphs. 
	In this section, this technique is formalised and generalised to the setting of a general partition of the vertex set of a graph, provided that the quantum isomorphism satisfies a simple condition stated below. Immediate examples of such partitions, beyond the case of connected components, are the partition obtained from a fixed subset and the connected components obtained by removing it and (by~\Cref{lem:qi_blocks}) the partition into the cut vertices of the graph and the internal vertices of distinct blocks of it. 
		
	A \textit{partitioned graph} is a pair $(G,\Pp)$ where $G$ is a graph and $\Pp$ is a partition of $V(G)$. 
	We denote by $\sim_\Pp$ the equivalence relation naturally associated to $\Pp$ on $V(G)$. 
	
	Given two partitioned graphs $(G,\Pp_G)$ and $(H,\Pp_H)$, we say that a quantum isomorphism $U$ from $G$ to $H$ is a \textit{quantum isomorphism of partitioned graphs} from $(G,\Pp_G)$ to $(H,\Pp_H)$ if, for all $x,y\in V(G)$ and $a$, $b\in V(H)$, we have $u_{ax}u_{by} = 0$ if $x \sim_{\Pp_G} y$ and $a\not\sim_{\Pp_H} b$, or $x\not \sim_{\Pp_G} y$ and $a\sim_{\Pp_H} b$. 
	
	Note that this is a weaker condition than that of a quantum isomorphism of coloured graphs, where one asks the coefficients indexed by different colours to be zero. 
	
	Let $U$ be a such a quantum isomorphism of partitioned graphs as above with coefficients in a unital $C^*$-algebra $X$. 
	Let $C\in \Pp_G$ and $D\in \Pp_H$. 
	For $x\in V(G)$ and $a\in V(H)$, set $p_D(x) = \sum_{d \in D} u_{dx}$ and $q_C(a) = \sum_{w\in C} u_{aw}$. 
	
	\begin{lemma}\label{lem:projections}
		For $C \in \Pp_G $ and $ D \in \Pp_H$, the functions $p_D \colon V(G) \to X$ and $q_C\colon V(H)\to X$ are projection-valued and constant on the cells of $\Pp_G$ and $\Pp_H$ respectively. 
	\end{lemma}
	\begin{proof}
		First, using the fact that rows and columns of $U$ are orthogonal, it is easy to see that $p_D$ and $q_C$ are projection-valued. 
		
		Second, let $x$, $y\in V(G)$ be in a common cell of $\Pp_G$. 
		We have:
		\begin{align*}
			p_D(x)(1-p_D(y)) &= \left(\sum_{d\in D} u_{dx}\right) \left(\sum_{d\in V(H)} u_{dy} - \sum_{d\in D} u_{dy}\right)\\
			&= \sum_{d\in D,d'\notin D} u_{dx}u_{d'y}\\
			&= 0
		\end{align*}
		since $U$ preserves the partitions. 
		This shows that $p_D(x) = p_D(x)p_D(y)$. 
		Taking adjoints, one obtains $p_D(x) = p_D(y)p_D(x)$. 
		Applying what precedes to $y$ and $x$, we reach $p_D(y) = p_D(y)p_D(x) = p_D(x)$, as desired. 
		
		Finally, notice that $U^*$ is a quantum isomorphism from $(H,\Pp_H)$ to $(G, \Pp_G)$ preserving the partitions. 
		Adding the magic unitary used in index, we have that $q_{C,U} = p_{C,U^*}$, which is constant on the cells of $\Pp_H$ by what precedes. 
		This concludes the proof.
	\end{proof}
	
	We can now define $p_{DC}(U) = \sum_{d\in D} u_{dx}$ for some $x\in C$, and $ q_{CD}(U) = \sum_{w\in C} u_{aw}$ for some $a\in D$. 
	In particular, we have $q_{DC}(U^*) = p_{DC}(U)$ and $p_{CD}(U^*) = q_{CD}(U)$. 
	We will remove the mention of $U$ when the quantum isomorphism is clear from the context.
	
	\begin{lemma}\label{lem:p_q}
		For any $C\in \Pp_G$ and $D\in \Pp_H$, we have $q_{CD}(U) = p_{DC}(U)$.
	\end{lemma}
	
	\begin{proof}
		Let $x\in C$ and $a\in D$. 
		We have:
		\begin{align*}
			q_{CD}(U)(1-p_{DC}(U)) &= \left(\sum_{w\in C} u_{aw} \right) \left(\sum_{d\in V(H)} u_{dx} - \sum_{d\in D} u_{dx}\right)\\
			&= \sum_{w\in C,d\notin D} u_{aw}u_{dx}\\
			&= 0
		\end{align*}
		since $a \not \sim_{\Pp_H} d$, $x\sim_{\Pp_G} w$, and $U$ preserves the partitions. 
		Therefore, $q_{CD}(U) = q_{CD}(U)p_{DC}(U)$. 
		Applying what precedes to $q_{DC}(U^*)$ and $p_{CD}(U^*)$, we reach
		\[
		p_{DC}(U) = q_{DC}(U^*) = q_{DC}(U^*)p_{CD}(U^*)
		= p_{DC}(U)q_{CD}(U) = q_{CD}(U)^* = q_{CD}(U),
		\]
		as desired. 
		This concludes the proof.
	\end{proof}
	
	If $S\subseteq V(G)$ and $T\subseteq V(H)$, we denote by $U[T,S]$ the submatrix of $U$ indexed by $T\times S \subseteq V(H) \times V(G)$. 
	We now reach the fundamental theorem for partitioned graphs.
	
	\begin{theorem}\label{thm:partitioned_graphs}
		Let $(G,\Pp_G)$ and $(H,\Pp_H)$ be two partitioned graphs and denote by $A_G$ and $A_H$ the adjacency matrices of $G$ and $H$ respectively. 
		Let $U$ be a quantum isomorphism of partitioned graphs from $(G,\Pp_G)$ to $(H,\Pp_H)$ with coefficients in a unital $C^*$-algebra $X$. 
		Take $C\in \Pp_G$ and $D\in \Pp_H$, let $W = U[D,C]$, and $p = p_{DC}(U)$. 
		Then:
		\begin{enumerate}
			\item the matrix $P(U) = \left(p_{KL}(U)\right)_{K\in \Pp_H,L\in \Pp_G}$ is a magic unitary with coefficients in $X$,
			\item if $p\neq 0$, then it is the unit of the $C^*$-subalgebra $Y$ of $X$ generated by the coefficients of $W$,
			\item if $p = 0$, then $W = 0$,
			\item if $p \neq 0$, then $W$ is a magic unitary with coefficients in $Y$,
			\item $A_H[D,D]W = WA_G[C,C]$.
		\end{enumerate}
		In particular, there is a bijection $\varphi \colon \Pp_G \to \Pp_H$ such that $U[\varphi(S),S]$ is a quantum isomorphism from $G[S]$ to $H[\varphi(S)]$ for every $S\in \Pp_G$.
	\end{theorem}
	\begin{proof}
		Let us start with item (1). 
		By~\Cref{lem:projections}, we have that $P$ is a well-defined matrix of projections. 
		Let us check that $P=P(U)$ sums to $1$ on columns. 
		For $L\in \Pp_G$, taking $x\in L$, we have
		\[\sum_{K\in \Pp_H} p_{KL}(U) = \sum_{K\in \Pp_H} \sum_{a\in K} u_{ax} = \sum_{a\in V(H)} u_{ax} = 1,\]
		so the sum on the columns of $P$ is 1. Let us check for the rows. 
		For every $K\in \Pp_H$ and every $L \in \Pp_G$, using~\Cref{lem:p_q}, we have $p_{KL}(U) = q_{LK}(U) = p_{LK}(U^*)$. Now, for $K \in \Pp_H$, we have
		\[\sum_{L\in \Pp_G} p_{KL}(U) = \sum_{L\in \Pp_G} p_{LK}(U^*) = 1,\]
		by the previous computation applied to $U^*$. 
		Hence $P$ is a magic unitary, which proves (1).
		
		Items (2) and (3) follow from the orthogonality of rows and columns of magic unitaries, from the construction of $p_{DC}$, and from~\Cref{lem:p_q}. 
		
		Item (4) follows immediately from~\Cref{lem:p_q} and item (2).
		
		Let us prove item (5). 
		By assumption, we have $UA_G = A_HU$, so computing by blocks, we have
		$$\sum_{L\in \Pp_G} U[D,L]A_G[L,C] =[U A_G]_{DC} = [A_H U]_{DC} = \sum_{K\in \Pp_H} A_H[D,K]U[K,C].$$ 
		Since $P$ has orthogonal rows and columns, we have that $pU[D,L] = 0$ when $L\neq C$, and $U[K,C]p = 0$ when $K\neq D$. 
		Since $A_G$ and $A_H$ are scalar-valued, multiplying the previous equation on the left and on the right by $p$, we obtain 
		$$(pWp)A_G[C,C] = A_H[D,D](pWp).$$ 
		Notice that $pWp = W$ by items (2) and (3). 
		Hence, we have $ A_H[D,D]W = WA_G[C,C]$, as desired.
		
		Finally, let us prove the last statement of the theorem. 
		Applying Lemma 2.20 from~\cite{Meunier2023} to the magic unitary $P$ we have a bijection $\varphi \colon \Pp_G \to \Pp_H$ such that $p_{\varphi(S)S} \neq 0$ for every $S\in \Pp_G$. 
		Let $S\in \Pp_G$, set $T = \varphi(S)$ and $W' = U[T,S]$. 
		By item (4), $W'$ is a magic unitary, and by item (5) we have $A_H[T,T]W' = W'A_G[S,S]$. 
		But $A_G[S,S] = \Adj(G[S])$ and $A_H[T,T] = \Adj(H[T])$. 
		Hence $W'$ is a quantum isomorphism from $G[S]$ to $H[T]$ with coefficients in $Y'$ where $Y'$ is the unital $C^*$-algebra generated by the coefficients of $W'$. 
		This concludes the proof.
	\end{proof}
	
	Applying~\Cref{thm:partitioned_graphs} to the partitions of the connected components of the two graphs allows one to find back the fact that quantum isomorphic graphs have two-by-two quantum isomorphic connected components, a result initially obtained by the second author in~\cite[Theorem 3.7]{Meunier2023}. We restate this result in our setting.  
	
	\begin{theorem}\label{lem:connected_components_partition}
		Let $U$ be a quantum isomorphism from $G$ to $H$. 
		Let $\Pp_G$ be the partition of $V(G)$ into the vertex sets of the connected components of $G$ and define $\Pp_H$ similarly. 
		Then $U$ is a quantum isomorphism of partitioned graphs from $(G,\Pp_G)$ to $(H,\Pp_H)$. 
		In particular, the connected components of $G$ and $H$ are in bijection and are two-by-two quantum isomorphic.
	\end{theorem}
	
	\begin{proof}
		The last assertion follows immediately from~\Cref{thm:partitioned_graphs} once $U$ is shown to be a quantum isomorphism of partitioned graphs. 
		For this, let $x$, $y\in V(G)$ and let $a$, $b\in V(H)$. 
		Assume that $x$ and $y$ are in a common connected component of $G$ and that $a$ and $b$ are not in a common connected component of $H$. 
		Then $d_G(x,y) < +\infty = d_H(a,b)$, so by~\Cref{lem:Fulton} we have that $u_{ax}u_{by} = 0$. 
		Similarly, $u_{ax}u_{by} = 0$ if $d_G(x,y) = +\infty > d_H(a,b)$. 
		This concludes the proof.
	\end{proof}

	\section{Operations on anchored graphs}\label{sec:anchored_graphs}
	
	In this section, we build our main technical operation, namely the operation $\Gamma$, which allows us to decompose a connected graph into a disjoint union of graphs with strictly less blocks. The main result of this section is~\Cref{lem:qi_for_anchored_graphs} which states that $\Gamma$ preserves quantum isomorphisms. This enables us to apply induction to obtain the main results of this article in the following sections.

	\subsection{Anchored graphs}
	In this subsection, we introduce \emph{anchored graphs} which provide the suitable formalism for the operation $ \Gamma $.

	\begin{definition}\label{def:anchored_graph}
		A \emph{connected anchored graph} is a pair $ (G, R) $ where $ G $ is a connected graph and $ R \subseteq V(G) $ either consists of a cut vertex of $ G $ or a block of $ G $. 
		An \emph{anchored graph} is a pair $(G,R) $ where $ G $ is a graph, $ R \subseteq V(G) $, and for every connected component $ G_i $ of $ G $, the pair $ (G_i, R \cap V(G_i)) $ is a connected anchored graph.
	\end{definition}
	
	Let $ (G, R) $ and $ (H,S) $ be anchored graphs. 
	An \emph{isomorphism of anchored graphs} from $ (G, R) $ to $ (H,S) $ is a graph isomorphism $\phi : V(G) \to V(H)$ from $ G $ to $ H $ that preserves the anchors, that is $ \phi(R) = S $. 
	A \emph{quantum isomorphism of anchored graphs} from $ (G, R) $ to $ (H,S) $ is a quantum isomorphism $ U $ from $G$ to $H$ which preserves the anchors, that is $ u_{ax} = 0 $ if $ x \in R $ and $ a \notin S $ or if $ x \notin R $ and $ a \in S $. 
	
	The following lemma can be seen as an anchored version of~\Cref{lem:connected_components_partition}.

	\begin{lemma}\label{lem:qi_for_disconnected_anchored_graphs}
		Let $(G,R)$ and $(H,S)$ be anchored graphs and let $U$ be a quantum isomorphism of anchored graphs from $(G,R)$ to $(H, S)$ with coefficients in a unital $C^*$-algebra $X$.
		Let $G_1, \dots, G_k$ be the connected components of $G$ and $H_1, \dots, H_l$ be the connected components of $H$. 
		Let $R_i = R \cap V(G_i)$ for $1\leq i\leq k$ and $S_j = S \cap V(H_j)$ for $1\leq j\leq l$. 
		Then $l = k$ and up to reordering, for every $1\leq i\leq k$, we have that $ W_i = U[V(H_i), V(G_i)] $ is a quantum isomorphism of anchored graphs from $(G_i,R_i)$ to $(H_i,S_i)$ with coefficients in the $C^*$-subalgebra of $X$ generated by the coefficients of $W_i$. 
	\end{lemma}
	
	\begin{proof}
		Let $\Pp_G$ be the partition of $V(G)$ into the sets of vertices of the connected components of $G$, and let $\Pp_H$ be similarly defined for $H$. 
		By~\Cref{lem:connected_components_partition}, $U$ is a quantum isomorphism of partitioned graphs from $(G,\Pp_G)$ to $(H,\Pp_H)$, so, by~\Cref{thm:partitioned_graphs}, we have that $k=l$ and that up to reordering, for every $ 1 \leq i \leq k $, we have that $W_i = U[V(H_i),V(G_i)]$ is a quantum isomorphism from $G_i$ to $H_i$. Let $ i \in \{1, \dots, k\}$, we claim that $W_i$ is a quantum isomorphism of anchored graphs from $(G_i,R_i)$ to $(H_i,S_i)$. 
		Indeed, let $x\in V(G_i)$ and $a\in V(H_i)$ such that $w_{ax} \neq 0$. 
		If $x\in R_i$, then $x\in R$ so since $U$ is a quantum isomorphism of anchored graphs and $u_{ax} = w_{ax} \neq 0$ we have $a\in R$. 
		Since $a\in V(H_i)$, we have $a\in S_i$, as desired. 
		Similarly, we obtain that if $a\in S_i$, then $x\in R_i$. 
		Hence $W_i$ is a quantum isomorphism of anchored graphs from $(G_i,R_i)$ to $(H_i,S_i)$, which concludes the proof.
	\end{proof}
	
	\subsection{Relations between blocks and cut vertices of a graph and its block graph} \label{subsec:lambda}
	
	The aim of this subsection is to define a function $\lambda_{G} $ from the set of blocks and cut vertices of $G$ to the set of blocks and cut vertices of the block graph of $G$ which will allow us to define the block structures of an anchored graph. To do so, we need some lemmas. 
	
	\begin{lemma}\label{lemforlambda1}
		Let $ G$ be a connected graph. 
		For every block $B$ of $\Bb(G)$, there exists a unique cut vertex $v\in V(G)$ such that $\{v\} = \bigcap_{b \in B} b$. Conversely, for every cut vertex $v$ of $G$, the set of all blocks $b$ of $G$ such that $ v \in b $ forms a block of $\Bb(G)$. 
	\end{lemma}
	\begin{proof} 
		First, since $B$ is a block of $\Bb(G)$, by definition, it has at least two vertices. Next, by~\Cref{prop:characterisation_blockgraph}, every two distinct $b,b' \in B$ have nonempty intersection as blocks of~$G$. Moreover, by item (1)~\Cref{prop:properties_of_blocks}, this intersection contains exactly one vertex. Therefore, $ \bigcap B = \bigcap_{b \in B} b $ has at most one element. Also, if $\bigcap B \neq \varnothing $, then the vertex $v \in V(G)$ in this intersection is contained in at least two blocks of $G$, thus it is not an internal vertex. Therefore, by item (4) of~\Cref{prop:properties_of_blocks}, it is a cut vertex of $G$.
		
		Hence, it is enough to show that $\bigcap B \neq \varnothing $.
		If $ \abs{B} = 2 $, then the result is immediate since the two elements of $B$ have exactly one vertex in common. So, for the rest of the proof, we assume that~$ \abs{B} \geq 3$. 
		
		Let $b_1, b_2, b_3 \in B$ be blocks of $G$. Set $\{v_i\} = b_i \cap b_{i+1}$ for $i \in \{1,2,3\}$ (where indices are modulo~3). 
		
		First, if for some $i$, we have $v_i = v_{i+1}$, then by choice of $v_j$'s, we have $ v_i \in b_i $ and $v_i = v_{i+1} \in b_{i+2}$. Since $ v_i \in b_{i+2} \cap b_i = \{v_{i+2}\}$, we have $v_i = v_{i+2} $, which implies $v_i = v_{i+1} = v_{i+2}$.
		
		Second, if $v_i$'s are all distinct, then find a path $P_i$ of $G$ entirely in $b_i$ from $v_{i-1}$ to $v_i$, which is possible because $b_i$ is connected. It is straightforward to see that $ P_2 P_3 P_1 $ (obtained by concatenation of the three paths and then removing the repetitions of $v_i$'s) is a cycle of $G$. This is a contradiction with item (3) of~\Cref{prop:properties_of_blocks}. 
		
		Therefore, there exists a common vertex $v$ such that for all distinct $b, b'\in B$, we have $b \cap b' = \{v\}$, hence $\bigcap B = \{v\}$, as desired.
		
		Conversely, for a cut vertex $v$ of $G$, let $ B = \{b \in \BB(G)  \mid v \in b\}$. We show that $B$ is a block of $\Bb(G)$. Note that $\abs{B} \geq 2$ and it induces a complete graph in $ \Bb(G) $. So, $B$ is inside a unique block $ B'$ of $\Bb(G)$. Let $\{v'\} = \cap_{b \in B'}{b}$, which by the first part of the lemma is a cut vertex of $G$. If $v' \neq v$, then every $b \in B$ contains both $v$ and $v'$, and since $\abs{B} \geq 2$, we have at least two blocks that intersect in more than one vertex, a contradiction with item (1) of~\Cref{prop:properties_of_blocks}. Therefore $v'=v$ and consequently by definition of $B$, we have $B'=B$, which implies that $B$ is a block of~$\Bb(G)$, as desired.
	\end{proof}

	\begin{lemma} \label{lemforlambda2}
		Let $G$ be a connected graph. Every block $b \in \BB(G) $ of $G$ containing at least two cut vertices of $G$ is a cut vertex in $\Bb(G)$. Conversely, every non-isolated cut vertex $b$ of $\Bb(G)$ is a block of $G$ that contains at least two cut vertices of $G$.
	\end{lemma}
	\begin{proof}
		Let  $b \in \BB(G) $ and let $v_1, v_2 \in b $ be two distinct cut vertices of $G$. Let $ b_1$ and $b_2$ be distinct blocks of $G$ such that $b_i \cap b = \{v_i\}$. An elementary argument such as the one in the proof of~\Cref{lemforlambda1} shows that $b_1 \cap b_2 = \varnothing $. Therefore there is no edge between $b_1$ and $b_2$ in $\Bb(G)$. Thus, there exist distinct blocks $B_1$ and $B_2$ of $\Bb(G)$ such that $b, b_1 \in B_1$ and $b, b_2 \in B_2$. Therefore, $b$ is a cut vertex of $\Bb(G)$ whose removal separates $b_1$ and $b_2$ in $ \Bb(G) $. 
		
		Conversely, let $ b $ be a cut vertex of $\Bb(G)$. Let $ B_1 $ and $B_2$ be two distinct blocks of $\Bb(G)$ containing $b$. Let $ \{v_1\} = \bigcap B_1$ and $ \{v_2\}=\bigcap B_2$. By~\Cref{lemforlambda1}, $v_1$ and $v_2$ are distinct cut vertices of $G$.
	\end{proof}
	
	\begin{lemma} \label{lemforlambda3}
		Let $G$ be a connected graph. Every block $b \in \BB(G)$ of $G$ containing exactly one cut vertex of $G$ is an internal vertex of $\Bb(G)$. Conversely, every internal vertex $b$ of $\Bb(G)$, as a block of $G$, contains exactly one cut vertex of $G$. 
	\end{lemma}
	\begin{proof}
		Let $ b\in \BB(G)$ be a block of $G$ containing exactly one cut vertex of $G$. Since $G$ is not 2-connected, $b$ is not an isolated vertex of $ \Bb(G) $. Moreover, by~\Cref{lemforlambda2}, we have that $b$ is not a cut vertex of $\Bb(G)$. Hence, by item (4) of~\Cref{prop:properties_of_blocks}, we have that $b$ is an internal vertex of~$ \Bb(G) $. 
		
		Conversely, if $b$ is an internal vertex of $\Bb(G)$, then let $B$ be the unique block of $\Bb(G)$ containing~$b$. By~\Cref{lemforlambda1}, setting $ \{v\} = \bigcap B $, we have that $v$ is a cut vertex of $G$ that is contained in~$b$. Moreover, by~\Cref{lemforlambda2}, we have that $b$ contains at most one cut vertex of $G$. This completes the proof.
	\end{proof}
	
	Let us now define the function $\lambda$. For a graph $F$, let $ \CC_1(F) = \{ \{v\} \mid v \in \CC(F) \}$, that is, $\CC_1(F)$ is the set of the singletons of all cut vertices of $F$. Set $ \AA(F) = \CC_1(F)  \cup \BB(F) $. Note that $ \AA(F) $ is the set of all possible anchors for $F$. 
	
	Now we can define a function $\lambda = \lambda_{G} : \AA(G) \to \AA(\Bb(G))$ as follows. 
	\begin{itemize}
		\item If $ b=\{r\} \in \CC_1(G)$, then define $\lambda(b) = \{b' \in V(\Bb(G)) \mid r \in b' \}$, which is a block of $\Bb(G)$ by~\Cref{lemforlambda1}.
		
		\item If $b \in \BB(G)$ and $b$ contains at least two cut vertices of $G$, then define $\lambda(b) = \{b\}$ which is in $\CC_1(\Bb(G))$ by~\Cref{lemforlambda2}.
		
		\item If $b \in \BB(G)$ and $b$ contains exactly one cut vertex of $G$, then define $\lambda(b)$ to be the unique block of $ \Bb(G) $ containing $b$, which exists since, by~\Cref{lemforlambda3}, $b$ is an internal vertex of $\Bb(G)$. 
		
		\item If $b \in \BB(G)$ and $b$ contains no cut vertices of $G$, then $G$ is 2-connected and $\Bb(G) = K_1$. In this case, define $\lambda(b) = V(\Bb(G))$. 
	\end{itemize}
	
	It is immediate from~\Cref{lemforlambda1,lemforlambda2,lemforlambda3} that the function $\lambda_G$ is well-defined. As a result, if $(G,R)$ is an anchored graph, then so is $(\Bb(G), \lambda_G(R))$.

	\subsection{Block structures of anchored graphs}
	
	A \emph{connected rooted graph} is a pair $ (F, r) $, where  $ F $ is a connected graph and $ r \in V(F) $. 
	Since all the graphs that we will consider as rooted graphs are connected, we refer to them simply as rooted graphs. 
	An \emph{isomorphism of rooted graphs} $\alpha$ from a rooted graph $ (F, r) $ to a rooted graph $ (F', r') $ is a graph isomorphism from $ F $ to $ F' $ such that $\alpha(r) = r'$.

	Let $ (G, R) $ be a connected anchored graph.
	We define the \emph{block graph} of $(G,R)$ to be the anchored graph $\Bb(G,R) = (\Bb(G), \lambda_G(R))$. This is well-defined by~\Cref{subsec:lambda}. 
	
	We define its \emph{block tree} $\Tt(G,R)$ to be the following rooted tree. If $R$ is a block of $G$, we set $ \Tt(G,R) = (\Tt(G),R) $. If $R=\{r\}$ where $r$ is a cut vertex of $G$, we set $ \Tt(G,R) = (\Tt(G),r) $. However, for simplicity and by abuse of notation, we may also denote the latter by $ (\Tt(G),R) $. 
	
	We recall that, as explained in~\Cref{sec:preliminaries}, there is a natural proper 2-colouring of $V(\Tt(G))$ into blocks and cut vertices of $G$, and we will refer to blocks as white vertices and cut vertices as black vertices of $\Tt(G)$. Recall that a leaf of a tree is a vertex of degree 1. From the definition, it is straightforward that the leaves of a block tree are white. 
	
	\begin{lemma}\label{lem:iso_coloured_trees}
		Let $(G,R)$ and $(H,S)$ be connected anchored graphs with the same number of vertices, and let $\alpha \colon \Tt(G,R) \to \Tt(H,S)$ be an isomorphism of rooted trees. Then, $\alpha$ preserves the colours. In particular, $R$ is a block of $G$ if and only if $S $ is a block of $H$.
	\end{lemma}
	
	\begin{proof}
		Note that in every connected 2-coloured bipartite graph, if the colour of one vertex $v$ is given, the colour of every other vertex is determined by the parity of its distance to $v$.
		
		First, assume that $\Tt(G)$ has at least one leaf $ l $. Hence, $\alpha(l)$ is a leaf in $\Tt(H)$. Now, by what precedes, the colour of each vertex in $\Tt(G)$ and in $\Tt(H)$ is determined by its distance to $l$ and $\alpha(l)$ respectively. Since both $l$ and $\alpha(l)$ are white and since $\alpha$ preserves distances, we have that $\alpha$ preserves the colours. In particular, since $\alpha$ maps the root of $\Tt(G)$ to the root of $\Tt(H)$, we have that the roots are of the same colour, which completes the proof. 
		
		Second, assume that $\Tt(G)$ has no leaves. So, $\Tt(G) = K_1 = \Tt(H)$. If the only vertex of $ \Tt(G) $ is black, then $G = K_1$. Since $G$ and $H$ have the same number of vertices, we have that $H=K_1$ as well. Hence, the only vertex of $ \Tt(H) $ is black, as desired. Else, if the only vertex of $\Tt(G)$ is white, then $G$ is a 2-connected graph, so in particular it has strictly more than one vertex. Therefore, $H$ is not $K_1$ and thus the only vertex of $\Tt(H)$ is white. This concludes the proof.
	\end{proof}

	\subsection{Operation $\Gamma$}
	
	Given a connected anchored graph $ (G, R) $, we want to define $\Gamma(G,R) = (G', R')$, another anchored graph, where the number of blocks of each connected component of $ G' $ is strictly less than that of $ G $ when $G$ has at least two blocks. 
	We start with some preliminary notions.

	Let $ G $ be a graph, we define the function $\rho = \rho_G : V(G) \to \PP(V(G)) $ that maps $ v $ to $ \{v\} $ if $ v $ is a cut vertex or an isolated vertex of $ G $ and that maps $v$ to the vertex set of the unique block containing $ v $ otherwise. 
	Note that we have $ v \in \rho(v) $.
	
	\begin{remark} \label{remark:rootedtoanchored}
		Let $G$ be a graph with connected components $G_1, G_2, \dots, G_k$ and let $R \subseteq V(G)$ be such that for all $i$, we have $\abs{R \cap G_i} = 1$, i.e.\ $(G_i, R \cap G_i)$ is a rooted graph. It is straightforward to see that $ \bigcup \rho_G(R) = \bigcup_{r\in R} \rho(r) $ intersects each connected component of $G$ either in a cut vertex or a block of that connected component. Thus, $(G, \bigcup \rho_G(R))$ is an anchored graph.
	\end{remark}

	Let $ G $ be a connected graph on at least two vertices and let $ r $ be a cut vertex of $ G $. 
	Let $C_1,\ldots,C_k$ be the vertex sets of the connected components of $G \setminus r$ (notice that $k\geq 2$). 
	For $1\leq i\leq k$, we set $G_i = G[C_i\cup\{r\}]$, that is, every $G_i$ is the graph induced on $C_i \cup \{r\}$. 
	Let $G' = \bigoplus_{i=1}^k G_i$ be the disjoint union of $G_i$'s.
	Notice that for $1\leq i\leq k$, $G_i$ is exactly one connected component of $G'$ and it contains a copy of $r$, that we denote by $r_i$. 
	We call $ G' $ the \emph{split} of $ G $ over $ r$ and we denote it by $ \text{split}(G,r) $, we refer to $ r_1, \dots, r_k $ as the copies of $ r$ in $ G'$. Finally, when $G=K_1$, we set $\text{split}(G,r) = G $. We use these notations throughout the rest of the paper when needed.

	Now we can define the operation $\Gamma$ on connected anchored graphs. 
	Recall that $G[X]$ denotes the subgraph of $G$ induced by $X \subseteq V(G)$. 
	\begin{definition}\label{def:Gamma}
		Let $(G,R)$ be a connected anchored graph. 
		We define
		$$ 
		\Gamma(G,R) = \begin{cases}
			\left(\text{split}(G,r), \bigcup_{i=1}^k\rho_{\text{split}(G,r)}(r_i)\right) & \text{ if } R = \{r\} \text{ consists of a cut vertex of } G \\
			\left( G \setminus E(R), \bigcup_{r\in R} \rho_{G \setminus E(R)}(r) \right) & \text{ otherwise.} 
		\end{cases}
		$$
	\end{definition}
	
	Let us detail this definition. Note that it follows from~\Cref{remark:rootedtoanchored} that $\Gamma(G,R)$ is an anchored graph. 
	
	Case 1: $R = \{ r \}$ where $ r $ is a cut vertex of $ G $. Using the notations before the definition, in this case, no $r_i$ is a cut vertex in the connected component $G_i$ of $\text{split}(G,r)$, so $\rho_{\text{split}(G,r)}(r_i)$ is a block $b_i$ of $G_i$ and $(G_i, b_i)$ is an anchored graph. In this case, we set $\Gamma(G,R)$ to be the disjoint union of these anchored graphs. See~\Cref{figGammaCase1}.

	\begin{figure}
		\centering
		\begin{tikzpicture}[scale=.65]	
			\begin{scope}
				\draw (0,-0.2) ellipse (.8cm and 2cm);
				\draw[rotate=45, shift={(1,-.4)}] (0,-0.2) ellipse (.8cm and 2cm);
				\draw[rotate=-45, shift={(-1,-.4)}] (0,-0.2) ellipse (.8cm and 2cm);
				\fill (0,1.3) circle (2pt) node[above] {};
				\fill (0,1.22)  node[above] {$r$};
				\draw[dashed] (0,1.3) circle (.2cm);
			\end{scope}
			
			\draw[->, thick] (3,0) -- (6,0) node[midway, above] {$\Gamma$};
			\node at (11,-3.5) {$\Gamma(G,\{r\})$};
			\node at (0,-3.5) {$(G,\{r\})$};
			
			\begin{scope}[shift={(8,0)}]		
				\begin{scope}[rotate=-25]
					\draw (0,0) ellipse (1cm and 2cm);
					\fill (0,1.6) circle (2pt) node[below left] {$r_1$};
					\draw[dashed] (0,0.8) ellipse (0.5cm and 1cm);
					\node at (0,-2.3) {$G_1$};
					\node at (0,-0.75) {$\rho(r_1)$};
				\end{scope}
				
				\begin{scope}[shift={(3,-.2)}, rotate=0]
					\draw (0,0) ellipse (1cm and 2cm);
					\fill (0,1.6) circle (2pt) node[below] {$r_2$};
					\draw[dashed] (0,0.8) ellipse (0.5cm and 1cm);
					\node at (0,-2.3) {$G_2$};
					\node at (0,-0.75) {$\rho(r_2)$};
				\end{scope}
				
				\begin{scope}[shift={(6,0)}, rotate=25]
					\draw (0,0) ellipse (1cm and 2cm);
					\fill (0,1.6) circle (2pt) node[below right] {$r_k$};
					\draw[dashed] (0,0.8) ellipse (0.5cm and 1cm);
					\node at (0,-2.3) {$G_k$};
					\node at (0,-0.75)  {$\rho(r_k)$};
				\end{scope}
			\end{scope}
		\end{tikzpicture}
		\caption{Operation $\Gamma$ when the anchor consists of a cut vertex. The anchors are denoted by dashed lines.} \label{figGammaCase1}
	\end{figure}
	
	Case 2: $ R $ is a block of $ G $. In this case, we define $ G' $ to be the graph obtained from $G$ by removing all edges in $R$, that is, $ G' = G \setminus \{ xy \in E(G) \mid x,y \in R \} $.  We also set $R' = \bigcup_{r\in R} \rho_{G \setminus E(R)}(r)  $ and define $\Gamma(G,R) = (G',R')$. See~\Cref{figGammaCase2}.

	\begin{figure}
		\centering
		\begin{tikzpicture}[scale=.9]	
			\begin{scope}
				\draw[dashed] (0,0) ellipse (.95cm and .8cm);
				\draw[rotate=20, shift={(1.7,-.3)}] (0,0) ellipse (1.2cm and .5cm);
				\draw[rotate=-20, shift={(1.7,.3)}] (0,0) ellipse (1.2cm and .5cm);			
				\draw[shift={(0,-1.4)}] (0,0) ellipse (.6cm and 1cm);
				
				\fill (.8,0) circle (2pt) node[right] {$a$};
				\fill (0,-.6) circle (2pt) node[below] {$d$};
				\fill (0,.6) circle (2pt) node[above] {$b$};
				\fill (-.8,0) circle (2pt) node[left] {$c$};
				\draw (.8,0) -- (0, -.6) -- (-.8,0) -- (.8,0) -- (0,.6) -- (-.8,0) ;
			\end{scope}
			
			\draw[->, thick] (3,0) -- (6.2,0) node[midway, above] {$\Gamma$};
			\node at (9,-3.5) {$\Gamma(G,R)$};
			\node at (0,-3.5) {$(G,R)$};
			
			\begin{scope}[shift={(8,0)}]	
				\begin{scope}[shift={(1.6,0)}]
					\draw[rotate=20, shift={(1.7,-.3)}] (0,0) ellipse (1.2cm and .5cm);
					\draw[rotate=-20, shift={(1.7,.3)}] (0,0) ellipse (1.2cm and .5cm);				
					\fill (.8,0) circle (2pt) node[left] {};
					\fill (.7,0) node[] {};
					\fill (.8,-.05) node[below left] {{$\rho(a)$}};
					\draw[dashed] (0.8,0) circle (.18cm);
				\end{scope}	
				\begin{scope}[shift={(.8,.6)}]
					\fill (0,.6) circle (2pt) node[above] {};
					\draw[dashed] (0, 0.6) circle (.18cm);
					\fill (0.1,.6) node[below] {{$\rho(b)$}};
				\end{scope}
				\begin{scope}[shift={(0,0)}]
					\fill (-.8,0) circle (2pt) node[above] {};
					\draw[dashed] (-.8, 0) circle (.18cm);
					\fill (-.8, -.1) node[below] {{$\rho(c)$}};
				\end{scope}
				\begin{scope}[shift={(.8,-.6)}]
					\fill (0,-.6) circle (2pt) node[above right] {};
					\draw[shift={(0,-1.4)}] (0,0) ellipse (.6cm and 1cm);
					\draw[dashed] (0, -.95) ellipse (.3cm and .5cm) ;
					\fill (0, -1.3) node[below] {{ $\rho(d)$}};
				\end{scope}
			\end{scope}
		\end{tikzpicture}
		\caption{Operation $\Gamma$ when the anchor is a block. The anchors are denoted by dashed lines.} \label{figGammaCase2} 
	\end{figure}

	Note that by construction, if $G \neq K_1$ is connected but not 2-connected, setting $ (G',R') = \Gamma(G,R) $, we have that every connected component of $ G' $ has strictly less blocks than~$G$.
	
	\begin{lemma}\label{lem:gamma_definition}
		Let $(G,R)$ be a connected anchored graph and assume that $R$ is a block of $G$. 
		Set $(G',R') = \Gamma(G,R)$. 
		Notice that $V(G) = V(G')$. 
		We have:
		\begin{enumerate}
			\item $ R \subseteq R' $,
			\item if $ v \in R' \setminus R $, then there exists $u \in R $ such that $ v $ and $ u $ are in a common block of $ G $, and the number of blocks of $ G $ that contain $ u $ is exactly 2, 
			\item if $ u \in R $ and the number of blocks of $ G $ that contain $ u $ is exactly 2, then for every $ v \in V(G)  $ such that $ u $ and $ v $ are in a common block of $ G $, we have $ v \in R' $. 
		\end{enumerate}
	\end{lemma}
	
	\begin{proof}
		Let $\rho = \rho_{G'}$. 
		The first item follows directly from the fact that for every $ v $, we have $v \in \rho(v)$. 
		
		If $ v \in R' $, then there exists $ u \in R $ such that $ v \in \rho(u) $. 
		Now, if $ v \in R' \setminus R $, then $ v \neq u $, and in particular we cannot have $ \rho(u) = \{u\} $. 
		That is, $ \rho(u) $ is the vertex set of the unique block $ B $ of $ G'$ containing $ u $. 
		Thus $ v \in B $ as well. 
		Moreover, $ \rho(u) \neq \{u\} $ implies that $ u$ is neither a cut vertex nor an isolated vertex of $G' $. 
		Firstly, since it is not a cut vertex of $ G' $, it is in at most 2 blocks of $ G $. Secondly, since it is not an isolated vertex of $ G' $, it is in at least two blocks of $ G $. 
		This completes the proof of (2). 
		
		Finally, let $ u $ be a vertex as described in the statement of (3). 
		Let $ B $ be the block of $ G $ other than $ R $ containing $ u $. 
		By removing the edges of $ R $, this vertex $ u $ is contained in exactly one block $ B $ of $ G' $. 
		Thus $ \rho(u) = V(B) $. 
		Now, let $ v $ be a vertex in a common block with $ u $. 
		If $ v \in R$, then the result follows from (1). 
		If $ v \in B $, then $ v\in \rho(u) \subseteq R' $ which completes the proof.
	\end{proof}

	Now let us consider two connected anchored graphs $(G,R)$ and $(H,S)$. Given a quantum isomorphism of anchored graphs $ U $ from $ (G, R) $ to $ (H,S) $, we build a quantum isomorphism of anchored graphs $ \Gamma(U) $ from $ \Gamma(G, R) $ to $ \Gamma(H, S) $. 
	We start with a preliminary lemma. 
	
	\begin{lemma} \label{lem:anchor_is_the_same}
		Let $U$ be a quantum isomorphism of anchored graphs from $(G,R)$ to $(H,S)$. 
		Then $U[S,R]$ is a quantum isomorphism from $G[R]$ to $H[S]$. 
		Moreover, $R$ is a block of $G$ if and only if $S$ is a block of $H$.
	\end{lemma}
	\begin{proof}
		By definition, $U = \Diag(U_1,U_2)$ is diagonal by block with $U_2 = U[S,R]$. 
		In particular, $U_2$ is a magic unitary, and since $U\Adj(G) = \Adj(H)U$, we have $U_2\Adj(G[R]) = U_2\Adj(G)[R,R] = \Adj(H)[S,S]U_2 = \Adj(H[S])U_2$, so $U_2$ is a quantum isomorphism from $G[R]$ to $H[S]$. 
		The second part of the statement follows from the fact that $U_2$ is a square matrix.  
		This concludes the proof.
	\end{proof}
	
	When $ R = \{ r \} $ consists of a cut vertex, by~\Cref{lem:anchor_is_the_same}, so does $S$. We write $S = \{s\}$. In this case, we have $ U = \Diag(U_0, 1) $ where $U_0$ is a quantum isomorphism from $G\setminus r$ to $H\setminus s$. Let $C_1, \dots, C_k$ be the vertex sets of the connected components of $G$ and let $D_1, \dots, D_k$ be the vertex sets of the connected components of $H$ (the number of connected components is the same by~\Cref{lem:connected_components_partition}). By~\Cref{lem:connected_components_partition,thm:partitioned_graphs}, we have a magic unitary $P(U_0) = (p_{ij})_{1\leq i,j \leq k}$ where, after fixing $w_j \in C_j$ for every $j \in \{1, \dots, k\}$, we have $p_{ij} = \sum_{a \in D_i} u_{aw_j}$.

	\begin{definition}
		Let $(G,R)$ and $(H,S)$ be connected anchored graphs and let $U$ be a quantum isomorphism of anchored graphs from $(G,R)$ to $(H,S)$.
		\begin{itemize}
			\item If $ R$ is a block of $G$, we define $\Gamma(U)= U$.
			\item Otherwise, letting $U=\Diag(U_0,1)$, we define $\Gamma(U)=\Diag(U_0, P(U_0))$. That is, after fixing $w_j \in C_j$ for every $j \in \{1, \dots, k\}$, we have
			$$ [\Gamma(U)]_{ax} = \begin{cases}
				u_{ax} & \text{ if } a \notin  \{s_1, \dots, s_k \} \text{ and } x \notin  \{r_1, \dots, r_k \}  \\
				\sum_{b \in D_i} u_{bw_j} & \text{ if } a= s_i, x = r_j \\
				0 & \text{ otherwise}.
			\end{cases} $$
		\end{itemize}
	\end{definition}
	
	It is straightforward to check that $\Gamma(U)$ is indeed a magic unitary. We now prove the main theorem of this section.

	\begin{theorem}\label{lem:qi_for_anchored_graphs}
		If $U$ is a quantum isomorphism from $(G,R)$ to $(H,S)$, then $\Gamma(U)$ is a quantum isomorphism from $\Gamma(G,R)$ to $\Gamma(H,S)$. 
	\end{theorem}
	
	\begin{proof}
		Set $ (G',R') = \Gamma(G,R) $ and $ (H', S') = \Gamma(H,S) $. 
		
		Let us start with the case where $R = \{r\}$ and $r$ is a cut vertex in $G$. 
		In this case, by~\Cref{lem:anchor_is_the_same}, we have that $S = \{s\}$, with $s$ a cut vertex in $H$. 
		Denote the copies of $ r $ in $ G' =\text{split}(G,r) $ by $ \{r_1, \dots, r_k\} $ and the copies of $ s $ in $ H'=\text{split}(H,s) $ by $ \{s_1, \dots, s_l\} $. 
		Let $ A $ and $ B $ denote the adjacency matrices of $ G $ and $ H $ respectively. We have 
		$$
		A = \begin{pmatrix}
			A_1 & 0 & \ldots & 0 &  &  ^t L_1\\
			0 & A_2 & \ddots & \vdots &  &  ^t L_2 \\
			\vdots & \ddots  & \ddots  & 0 &  &  \vdots \\
			0  & \ldots & 0 & A_k & & ^tL_k \\
			& & & & & \\
			L_1 & L_2 & \ldots & L_k & & 0 
		\end{pmatrix},
		$$
		where $ A_j $ is the adjacency matrix of $ C_j $ (the $j$th connected component of $ G\setminus R$) and the last row and the last column of $ A $ are indexed by $ r $. 
		Setting $A_0 = \Diag(A_1, \dots, A_k) \text{ and } L = \Diag(L_1, \dots, L_k)$, we have that the adjacency matrix of $ G' $ is
		$$
		A' = \Adj(G') = \begin{pmatrix}
			A_0 & {}^tL \\
			L & 0
		\end{pmatrix}.
		$$
		Similarly, we have
		$$
		B = \begin{pmatrix}
			B_1 & 0 & \ldots & 0 &  &  ^t K_1\\
			0 & B_2 & \ddots & \vdots &  &  ^t K_2 \\
			\vdots & \ddots  & \ddots  & 0 &  &  \vdots \\
			0  & \ldots & 0 & B_l & & ^tK_l \\
			& & & & & \\
			K_1 & K_2 & \ldots & K_l & & 0
		\end{pmatrix},
		$$
		where $ B_i $ is the adjacency matrix of $ D_i $ (the $i$th connected component of $ H\setminus S$) and the last row and the last column of $ B $ are indexed by $ s $. 
		Again, setting $B_0 = \Diag(B_1, \dots, B_l) \text{ and } K = \Diag(K_1, \dots, K_l)$, we get that the adjacency matrix of $ H' $ is
		$$
		B' = \Adj(H') = \begin{pmatrix}
			B_0 & {}^t K \\
			K & 0
		\end{pmatrix}.
		$$
		
		Recall that $U = \Diag(U_0, 1)$. Using the notations from the definition of $ \Gamma(U) $, we have $ V = \Gamma(U) = \Diag(U_0 , P(U_0)) $,
		where
		$$
		U_0 =  \begin{pmatrix}
			U_{11}  & \ldots & U_{1k}  \\
			\vdots  & \ddots  & \vdots \\
			U_{l1}   & \ldots  & U_{lk} 
		\end{pmatrix}.
		$$
		In particular, since $U_0$ is a magic unitary, we have $k=l$. 
		Here, $ U_{ij} $ is the submatrix of $ U $ indexed by rows corresponding to $ D_i $ and columns corresponding to $ C_j $, that is $ U_{ij} = U[D_i, C_j]$. 
		Let $ P = P(U_0) $, it is a $ k \times k $ matrix whose entry in row $ i $ and column $ j $ is defined to be $p_{ij} = \sum_{z \in D_i} u_{zw_i} $. 
		Remember that by~\Cref{thm:partitioned_graphs} it is equal to the sum of any row or any column of $U_{ij}$. 
		We have
		$$ VA' = \begin{pmatrix}
			U_0 & 0 \\ 0 & P
		\end{pmatrix} \begin{pmatrix}
			A_0 & {}^tL \\ L & 0
		\end{pmatrix} = \begin{pmatrix}
			U_0 A_0 & U_0 {}^tL \\ PL & 0 
		\end{pmatrix},$$
		and
		$$ B'V = \begin{pmatrix} B_0 & {}^tK \\ K & 0
		\end{pmatrix} \begin{pmatrix} U_0 & 0 \\ 0 & P
		\end{pmatrix} = \begin{pmatrix} B_0 U_0 & {}^tK P \\ KU_0 & 0
		\end{pmatrix}. $$
		Since $ BU= UA$, we have $B_0U_0 = U_0A_0$. 
		Now, let us verify that $KU_0 = PL$. 
		Let $1\leq i,j \leq k$. 
		We have $$(KU_0)[\{s_i\},C_j] = K_iU_{ij} \text{ and } (PL)[\{s_i\},C_j] = p_{ij}L_j.$$ 
		So, we need to prove that $K_i U_{ij} = p_{ij}L_j$. 
		Notice that $BU = UA$, so 
		$\begin{pmatrix} L_1 & \ldots & L_k \end{pmatrix} = \begin{pmatrix} K_1 & \ldots & K_k \end{pmatrix} U_0 $. 
		In particular, we have that $L_j = \sum_{m=1}^k K_m U_{mj} $. 
		Multiplying by $p_{ij}$ on the left and using the fact that $K_m$ is scalar-valued for $1\leq m\leq k$, we obtain that $p_{ij}L_j = \sum_{m=1}^k K_m p_{ij}U_{mj}$. 
		Since $p_{ij}U_{mj} = \delta_m^i U_{ij}$, we reach $p_{ij}L_j = K_iU_{ij}$, as desired. 
		This being true for all $1\leq i,j\leq k$, we have that $KU_0 = PL$. 
		
		Now notice that $P$ and $U_0$ are magic unitaries, so in particular they are unitaries. 
		So, we multiply $ KU_0 = PL $ on the left by $ P^{-1} = {}^tP $ and from the right side by $ U_0^{-1} = {}^tU_0 $. 
		We get
		$$ {}^tP K = {}^tP (KU_0) {}^tU_0 = {}^tP(PL){}^tU_0 = L{}^tU_0.$$
		Therefore, we have ${}^tK P = U_0{}^tL$, as required. 
		
		This completes the proof that $B'V = VA'$, hence $V$ is a quantum isomorphism from $H'$ to $G'$. 
		It remains to prove that $V$ preserves $R'$ and $S'$. 
		Recall that by definition $R'$ is the union of the blocks containing the $r_i$, and similarly for $S'$. 
		Let $x\in V(G')$ and $a\in V(H')$ such that $v_{ax}\neq 0$. 
		Assume that $a\in S'$. 
		Then there exists $1\leq i\leq k$ such that $a$ in is the same block as $s_i$. 
		We have $0 \neq v_{ax} = \sum_{j=1}^k v_{ax}v_{s_ir_j}$, so there exists $1\leq j\leq k$ such that $v_{ax}v_{s_ir_j} \neq 0$. 
		Since $a$ and $s_i$ are in a common block of $G'$, by~\Cref{lem:qi_blocks}, we have that $x$ and $r_j$ are in a common block. 
		Hence $x\in R'$. 
		Conversely, if $x\in R'$, then we show in the same way that $a\in S'$. 
		This shows that $V$ is a quantum isomorphism of anchored graphs, as desired.
		
		Let us now consider the case where $R$ is a block of $G$. 
		We have that $S$ is also a block of $H$ by~\Cref{lem:anchor_is_the_same}. 
		Let $R^\circ $ (resp.\ $S^\circ$) be the vertices in $R$ (resp.\ in $S$) that are not cut vertices in $G$ (resp.\ in $H$), so $R^\circ$ and $S^\circ$ can be empty. 
		Using the fact that $U$ is a quantum isomorphism of anchored graphs and~\Cref{lem:qi_blocks}, we obtain that $U$ is diagonal by blocks $U = \Diag(U_1,U_2,U_3)$ where $U_1 = U[S^\circ,R^\circ]$, $U_2 = U[S\setminus S^\circ, R\setminus R^\circ]$, and $ U_3 = U[V(H)\setminus S,V(G)\setminus R]$. 
		The adjacency matrix of $ G $ is of the form 
		$$ A = \begin{pmatrix}
			A_1 & ^tK & 0  \\
			K & A_2 & ^t L \\
			0 & L & A_3
		\end{pmatrix}, $$ 
		where $ A_1$, $ A_2 $, and $ A_3 $ are respectively the adjacency matrices of $ G[R^\circ] $, $ G[R\setminus R^\circ] $, and $ G \setminus R $. Similarly, the adjacency matrix of $ H $ is of the form $$ B = \begin{pmatrix}
			B_1 & ^tM & 0  \\
			M & B_2 & ^t N \\
			0 & N & B_3
		\end{pmatrix}, $$ 
		where $ B_1$, $ B_2 $, and $ B_3 $ are respectively the adjacency matrices of $ H[S^\circ] $, $ H[S\setminus S^\circ] $, and $ H \setminus S $. Therefore, the adjacency matrices of $ G' $ and $ H' $ are respectively as follows: 
		$$
		A' = \Adj(G') =  \begin{pmatrix}
			0 & 0 & 0  \\
			0 & 0 & ^t L \\
			0 & L & A_3
		\end{pmatrix} \text{ and }
		B' = \Adj(H') = \begin{pmatrix}
			0 & 0 & 0  \\
			0 & 0 & ^t N \\
			0 & N & B_3
		\end{pmatrix}.
		$$
		Now, recall that $\Gamma(U) = U$. 
		Since $BU= UA$ and $U = \Diag(U_1,U_2,U_3)$, it is immediate to check that $B'U = UA' $. 
		Let us check that $U$ preserves $R'$ and $S'$, that is, $u_{ax} = 0$ if $x \in R'$ and $a \notin S'$ or $x \notin R'$ and $a \in S'$. 
		
		Let $a\in V(H)$ and $x\in V(G)$ such that $u_{ax}\neq 0$. 
		First, assume that $x\in R'$. 
		Two cases are possible. 
		If $x\in R$, then since $u_{ax}\neq 0$ we have by assumption that $a\in S$. 
		Since $S\subseteq S'$ by~\Cref{lem:gamma_definition}, we obtain that $a\in S'$, as desired. 
		Otherwise, $x\in R'\setminus R$ and, by~\Cref{lem:gamma_definition}, $x$ is in a common block in $G$ with a vertex $y\in R$ such that $y$ is contained in exactly one more block in $G$ other than $R$. 
		Now, we have $0 \neq u_{ax} = u_{ax}\left(\sum_{c\in V(H)} u_{cy}\right) = \sum_{c\in V(H)} u_{ax}u_{cy}.$ 
		Hence there exists $b\in V(H)$ such that $u_{ax}u_{by}\neq 0$. 
		Since $x$ and $y$ are in a common block in $G$, by~\Cref{lem:qi_blocks}, $a$ and $b$ are in a common block in $H$. 
		Moreover, we have $u_{by}\neq 0$ and $y\in R$, so by assumption we obtain that $b\in S$. 
		Since by what precedes $U$ is a quantum isomorphism from $G'$ to $H'$ and $y$ is not a cut vertex in $G'$, by~\Cref{lem:qi_blocks}, $b$ is not a cut vertex in $H'$ neither. 
		Hence $b$ is contained in exactly one more block in $H$ than $S$. 
		So by~\Cref{lem:gamma_definition}, we have that $a\in S'$, as desired.
		
		Second, applying what precedes to $U^*$ as a quantum isomorphism from $(H,S)$ to $(G,R)$, assuming that $u_{ax} = [U^*]_{xa} \neq 0$ and that $a\in S'$, we obtain that $x\in R'$, as desired. 
		This shows that $U = \Gamma(U)$ is a quantum isomorphism of anchored graphs from $(G,R)$ to $(H,S)$ and concludes the proof.
	\end{proof}

	\subsection{Operations $\Delta_i$}
	Finally, we introduce operations $\Delta_1$ and $\Delta_2$ that allow us to recover the block structure of a connected anchored graph $(G,R)$ knowing the block structure of $\Gamma(G,R)$ (see~\Cref{lem:Gamma_Delta_tree}). This provides the final tool to complete the inductive step in some proofs in the next section. 
	
	Let $ (T_1, r_1), \dots, (T_k, r_k) $ be rooted 2-coloured trees (let us call the colours black and white). 
	We define two operations $\Delta_1$ and $\Delta_2$, which given these rooted coloured trees, return another rooted coloured tree.

	\emph{Operation $\Delta_1$.} 
	This operation is defined if every $ r_i $ is white. 
	In that case, we first consider the disjoint union $ T= \bigoplus_{i=1}^k T_i $, then we add a new black vertex $ r $ and join it to $ r_1, \dots, r_k $. 
	The rooted 2-coloured tree $ (T,r) $ obtained is $ \Delta_1 \left((T_1, r_1), \dots, (T_k, r_k)\right) $. 
	
	\emph{Operation $\Delta_2$.} 
	For this operation, the color of $ r_i$'s can be arbitrary. 
	Among $ (T_i, r_i) $'s, let $ \mathcal A = \{(A_1, a_1), \dots, (A_l, a_l)\} $ be the rooted trees whose roots are white. 
	Also, let $ \mathcal B = \{(B_1, b_1), \dots, (B_m, b_m)\} $ be the rooted trees on at least 2 vertices whose roots are black. 
	(The rooted trees on exactly 1 vertex and whose roots are black do not play a role in the definition of $\Delta_2 $.)
	First, for each $ 1 \leq i \leq l $, add a new new black vertex $ d_i $ to $ (A_i, a_i) $, join it to $ a_i $ and set $d_i $ to be the new root to define a new rooted 2-coloured tree $ (D_i, d_i) $. 
	Then, set $ T_0 = (\bigoplus_{i=1}^l D_i) \oplus ( \bigoplus_{i=1}^m B_i) $. 
	Add a new white vertex $ r$ to $ T_0$, join it to $ d_1, \dots, d_l, b_1, \dots, b_m $ to create a 2-coloured tree $T$. 
	Set $ r $ as the new root. 
	We define $ \Delta_2\left((T_1, r_1), \dots, (T_k, r_k)\right) = (T,r) $.
	
	Now, let us explain the connection between the operations $ \Delta_i $ and $\Gamma $. 
	Recall that we have set a 2-colouring of block trees in \Cref{sec:preliminaries} were the blocks are white and the cut vertices are black.
	
	\begin{lemma} \label{lem:Gamma_Delta_tree}
		Let $ (G,R) $ be a connected anchored graph on at least two vertices. Set $i=1$ if $R$ consists of a cut vertex of $G$ and $i=2$ otherwise. Let $(G_1, R_1), \dots, (G_k, R_k)$ be the connected components of $\Gamma(G,R)$, and set $(T_j, r_j) = \Tt(G_j, R_j)$ for $j \in \{1, \dots, k\}$. 
		Then, we have $ \Tt(G,R) = \Delta_i\left((T_1,r_1), \dots, (T_k,r_k)\right) $.
	\end{lemma}
	
	\begin{proof}
		Let $(G',R') = \Gamma(G,R)$. 
		
		First, if $R= \{v\}$ consists of a cut vertex of $G$, then $G' = \text{split}(G,v)$. Since $G$ has at least two vertices, then $k\geq 2$ and each $v_j$ is an internal vertex in $G_j$. Thus, $R_j = \rho(v_j)$ is a block of $G_j$. Therefore, $r_j$'s are white. The rest follows from the construction of $\Delta_1$. 
		
		Second, assume that $R$ is a block of $G$. Note that the connected components of $G'$ are in one-to-one correspondence with the vertices of $R$. Indeed, $R \cap V(G_j)$ is a singleton. We denote this vertex by $v_j$. Recall that $(T_j, r_j) = \Tt(G_j, R_j)$. It is straightforward to see that:
		\begin{itemize}
			\item[(i)] $r_j$ is white if and only if $v_j$ is an internal vertex in $G'$ if and only if $v_j$ is a cut vertex in $G$ that is contained in exactly two blocks of $G$,
			\item[(ii)] $r_j$ is black and $T_j$ has at least two vertices if and only if $v_j$ is a cut vertex in $G'$ if and only if $v_j$ is a cut vertex in $G$ contained in at least three blocks,
			\item[(iii)] $r_j$ is black and $T_j$ has one vertex if and only if $v_j$ is an isolated vertex in $G'$ if and only if it is an internal vertex of $G$.
		\end{itemize}
		See~\Cref{figGammaCase2}. The proof then follows from the construction of $\Delta_2$. 
	\end{proof}
	
	Finally, the next lemma allows us to obtain global isomorphisms between block structures. 
	
	\begin{lemma}\label{lem:extending_alphai_to_alpha}
		Let $ (T_1, r_1), \dots, (T_k, r_k) $ and $ (T'_1, r'_1), \dots, (T'_k, r'_k) $ be rooted 2-coloured trees. 
		Assume that for each $ i \in \{1, \dots, k\} $, there exists an isomorphism $ \alpha_i : (T_i, r_i) \to (T'_i, r'_i)  $ which preserves the roots and the colours. 
		\begin{enumerate}
			\item If every $ r_i$ is a white vertex, setting $ (T,r) = \Delta_1\left((T_1, r_1), \dots, (T_k, r_k)\right) $ and $ (T',r') = \Delta_1\left((T'_1, r'_1), \dots, (T'_k, r'_k)\right) $, then the extension $\alpha:  (T,r) \to  (T',r') $ of $ \alpha_i$'s defined as below is a graph isomorphism which preserves roots and colours:
			$$
			\alpha(v) = \begin{cases}
				\alpha_i(v) & \text{ if }  v \in V(T_i) \\
				r' & \text{ if } v=r.
			\end{cases}
			$$
			\item Setting $ (T,r) = \Delta_2\left((T_1, r_1), \dots, (T_k, r_k)\right) $ and $ (T',r') = \Delta_2\left((T'_1, r'_1), \dots, (T'_k, r'_k)\right) $ and using the notations from the definition of $\Delta_2 $, the extension $\alpha: (T,r) \to  (T',r') $ of $ \alpha_i$'s defined as below is a graph isomorphism which preserves roots and colours:
			$$
			\alpha(v) = \begin{cases}
				\alpha_i(v) &  \text{ if } v \in V(T_i) \\
				d_i' & \text{ if } v= d_i \\
				r' & \text{ if } v=r. 
			\end{cases}
			$$
		\end{enumerate}
	\end{lemma}
	
	\begin{proof}
		The proof is immediate from the definitions of $\Delta_1$ and $\Delta_2$. 
	\end{proof}
	
	\section{Block structures of quantum isomorphic graphs}\label{sec:pegah_conjecture}
	
	In this section we prove the main result of the paper. We first prove the
	result for anchored graphs and then deduce it for graphs.

	\begin{lemma} \label{lem:anchored_qiso_implies_blocks_iso}
		Let $ (G, R) $ and $ (H, S) $ be two quantum isomorphic connected anchored graphs. Then:
		\begin{enumerate}
			\item there exists $ \alpha: \Tt(G, R) \rightarrow \Tt(H, S) $ such that
			\begin{itemize}
				\item $\alpha$ is an isomorphism between the block trees of $ (G,R) $ and $ (H,S) $ preserving the roots and the colours,
				\item for every block $ b $ of $ G $, we have $ b \qi \alpha(b) $,
			\end{itemize}
			\item there exists $ \beta: \Bb(G, R) \rightarrow \Bb(H, S) $ such that 
			\begin{itemize}
				\item $\beta$ is an isomorphism between the block graphs of $ (G,R) $ and $ (H,S) $ preserving the anchors,
				\item for every block $ b $ of $ G $, we have $ b \qi \beta(b) $.
			\end{itemize}
		\end{enumerate}
	\end{lemma}
	
	\begin{proof}
		If $\abs{V(G)} = 1$, then the result is immediate. So, we may assume that $\abs{V(G)}>1$, which implies that $G$ has at least one block. 
		
		Let us start by proving Item (1) by induction on the number of blocks of $G$ when it has at least one block. 
		If $ G $ has exactly one block, i.e.\ it is 2-connected, then by~\Cref{coro:qi_2connected}, $H$ is 2-connected as well. Therefore, the block structures are trivial and the result follows immediately. 
		
		Assume now that the result holds for every connected anchored graph with at most $n$ blocks where $n \geq 1$, and assume that $G$ has $n+1$ blocks. Set $(G',R') = \Gamma(G, R)$ and $(H',S') = \Gamma(H,S)$. 
		By~\Cref{lem:qi_for_anchored_graphs}, $V = \Gamma(U)$ is a quantum isomorphism from $(G',R')$ to $(H',S')$. 
		Recall that these two anchored graphs are not connected. 
		We write $(G',R') = \bigoplus_{i=1}^k (G_i,R_i)$ and $(H', S') = \bigoplus_{i=1}^{k'} (H_i, S_i) $. 
		By~\Cref{lem:qi_for_disconnected_anchored_graphs}, up to reordering, we have that $k=k'$ and that $(G_i, R_i)$ is quantum isomorphic to $(H_i, S_i)$. 
		By induction hypothesis, for every $ i \in \{1, \dots, k\}$ there exists a colour-preserving isomorphism of rooted trees $\alpha_i: \Tt(G_i, R_i) \to \Tt(H_i, S_i) $ such that for every block $ b $ of $ G_i $, we have $ b \qi \alpha_i(b) $.
		So, by~\Cref{lem:extending_alphai_to_alpha}, we can find an isomorphism preserving the roots and the colours
		$$\alpha^{(j)}: \Delta_j(\Tt(G_1, R_1), \dots, \Tt(G_k, R_k)) \to \Delta_j(\Tt(H_1, S_1),\dots, \Tt(H_k, S_k)) $$ for $ j =1 $ when $\Delta_1$ is defined and for $ j=2 $ otherwise.
		
		Now, if $ R = \{r\}$ where $r$ is a cut vertex of $ G $, then by~\Cref{lem:anchor_is_the_same}, $ S = \{s\} $ also consists of a cut vertex of $ H $. 
		So by applying~\Cref{lem:Gamma_Delta_tree} twice, we have that $\Delta_1$ is defined in both equations below, and that we have: 
		$$ \Tt(G, R) = \Delta_1(\Tt(G_1, R_1),\dots, \Tt(G_k, R_k)) \text{ and } \Tt(H, S) = \Delta_1(\Tt(H_1, S_1),\dots, \Tt(H_k, S_k)). $$ 
		Thus $\alpha = \alpha^{(1)} $ is the desired isomorphism. Moreover, in this case, for every block $b$ of $G$, there exists $i \in \{1, \dots, k\}$ such that $b$ is a block of $G_i$. Hence, we have $\alpha(b) = \alpha_i(b) \qi b $,
		where the last relation follows from the induction hypothesis. 
		
		If $R$ is a block of $G$, then $S$ is a block of $H$ by~\Cref{lem:anchor_is_the_same}. 
		Moreover, by~\Cref{lem:Gamma_Delta_tree}, we have 
		$$ \Tt(G, R) = \Delta_2(\Tt(G_1, R_1),\dots, \Tt(G_k, R_k)) \text{ and } \Tt(H, S) = \Delta_2(\Tt(H_1, S_1),\dots, \Tt(H_k, S_k)). $$ 
		In this case, $\alpha = \alpha^{(2)} $ is the desired isomorphism. 
		Moreover, if $ b $ is a block of $ G $, then either $ b = R $ or there exists $i \in \{1, \dots, k\}$ such that $b$ is a block of $G_i$. In the former case, by~\Cref{lem:anchor_is_the_same}, we have $ \alpha(b) = \alpha(R) = S \qi R = b $, as desired. In the latter case, using the induction hypothesis, we have $ \alpha(b) = \alpha_i(b) \qi b$. 
		
		This concludes the proof of Item (1).
		
		Let us now prove Item (2). Note that $V(\Bb(G)) = \BB(G)  \subseteq V(\Tt(G))$, so we define $ \beta : \Bb(G, R) \rightarrow \Bb(H, S) $ by setting $\beta(b) = \alpha(b)$ for every block $b$ of $G$. 
		
		First of all, note that for every block $b$ of $G$, we have $ b \qi \alpha(b) = \beta(b) $. It remains to prove that $\beta$ is an isomorphism preserving the anchors.
		
		Now, let us verify that $\beta$ is a graph isomorphism from $ \Bb(G) $ to $ \Bb(H) $. 
		Note that $\alpha$ induces a bijection between the blocks of $G$ and the blocks of $H$, hence $\beta$ is a bijection, whose inverse is given by $\beta^{-1}(b) = \alpha^{-1}(b)$. 
		Now, let $ bb' \in E(\Bb(G)) $. So, there exists a cut vertex $ c $ of $ G $ which is in the intersection of $ b $ and $ b' $. That is, in $\Tt(G)$, the vertex $c$ is a common neighbour of the vertices $b$ and $b'$. Therefore, in $ \Tt(H) $, the vertex $\alpha(c)$ is a common neighbour of $ \alpha(b) = \beta(b) $ and $\alpha(b') = \beta(b')$. Thus, the blocks $\beta(b)$ and $\beta(b')$ have nonempty intersection, so $\beta(b)\beta(b') \in E(\Bb(H))$. This shows that $\beta$ is a graph morphism. By the same argument, using $\alpha^{-1}$ instead of $\alpha$, we conclude that $\beta^{-1}$ is a graph morphism as well. Hence, we have proved that $\beta$ is a graph isomorphism. 
		
		Finally, let $ \Bb(G, R) = (\Bb(G), \lambda_{G}(R)) $ and $ \Bb(H, S) = (\Bb(H), \lambda_{H}(S)) $. In order to prove that $\beta$ preserves the anchors, we need to show that $ \beta(\lambda_G(R)) = \lambda_H(S) $. There are four possible cases. 
		
		Case 1: $R= \{r\}$ consists of a cut vertex of $G$, in which case, by~\Cref{lem:anchor_is_the_same}, $S=\{s\}$ also consists of a cut vertex of $H$.
		
		Moreover, in this case, by definition of $\lambda$ in Section 4, we have 
		$$ \lambda_G(R) = \{ b \in V(\Bb(G))  \mid r \in b  \} = N_{\Tt(G)}(r) \text{ and } \lambda_H(S) = \{ b \in V(\Bb(H))  \mid s \in b  \} = N_{\Tt(H)}(s). $$
		Since $r$ and $s$ are roots of $\Tt(G,R)$ and $\Tt(H, S)$ respectively and since $ \alpha $ preserves the roots, we have 
		$$\beta(\lambda_G(R)) = \alpha(N_{\Tt(G)}(r)) = N_{\Tt(H)}(\alpha(r)) = N_{\Tt(H)}(s) = \lambda_H(S),$$ as desired. 
		
		Case 2: $R$ is a block of $G$ and it contains no cut vertices of $G$. In this case, $G$ is a 2-connected graph, and so is $H$ by~\Cref{coro:qi_2connected}. Thus the result is immediate. 
		
		Case 3: $R$ is a block of $G$ containing exactly one cut vertex of $G$, in which case $S$ is also a block of $H$, by~\Cref{lem:anchor_is_the_same}. Moreover, by construction, $R$ and $S$ are the roots of $\Tt(G)$ and $\Tt(H)$ respectively. This implies that $\beta(R) = \alpha(R) = S$. Furthermore, $R$ is a leaf of $\Tt(G)$, and since $\alpha$ is an isomorphism, $S$ is also a leaf in $\Tt(H)$. Therefore, $S$ is a block of $H$ containing exactly one cut vertex of $H$.
		
		Now, by~\Cref{lemforlambda3}, we have that $R$ and $S$ are internal vertices of $\Bb(G)$ and $\Bb(H)$ respectively. Thus each of them is contained in a unique block, namely $\lambda_G(R)$ and $\lambda_H(S)$ respectively. This entails that 
		$ \lambda_G (R) = N_{\Bb(G)}[R] $, the closed neighbourhood of $R$ in $\Bb(G)$, and similarly, $ \lambda_H(S) = N_{\Bb(H)}[S] $. Therefore, we have
		$$
		\beta(\lambda_G(R)) = \beta(N_{\Bb(G)}[R]) = N_{\Bb(H)}[\beta(R)] = N_{\Bb(H)}[S] = \lambda_H(S),
		$$
		as desired.
		
		Case 4: $R$ is a block of $G$ containing at least two cut vertices of $G$. Since $R$ and $S$ are roots of the respective block trees, and since $\alpha$ preserves the roots, we have $\beta(R) = \alpha(R) = S$. Moreover, in $\Tt(G)$, the vertex $R$ has at least two neighbours, which implies that $S$ has at least two neighbours in $\Tt(H)$ as well. Now, both $R$ and $S$ contain at least two cut vertices, so by~\Cref{lemforlambda2}, they are cut vertices in $\Bb(G)$ and $\Bb(H)$ respectively. Hence, we have $\lambda_{G}(R) = \{R\}$ and $\lambda_{H}(S) = \{S\}$, which implies that 
		$$
		\beta(\lambda_G(R)) = \beta(\{R\}) = \{\beta(R)\}= \{S\} = \lambda_H(S),
		$$
		as desired. This concludes the proof of the theorem. 
	\end{proof}

	Recall from~\Cref{sec:preliminaries} that for a connected graph $G$, we denote by $Z(G)$ the center of $G$, and that we define $\bZ(G)$ to be $Z(G)$ if $Z(G)$ consists of a single cut vertex, and to be the unique block containing $Z(G)$ otherwise. In particular, $(G, \bZ(G))$ is an anchored graph. 
	
	The following lemma allows us to apply~\Cref{lem:anchored_qiso_implies_blocks_iso} to general graphs and obtain our main result.

	\begin{lemma} \label{lem:center-s-block-is-preserved} 
		Let $U$ be a quantum isomorphism from a connected graph $G$ to a graph $H$. Then $U$ is a quantum isomorphism of anchored graphs from $(G, \bZ(G))$ to $(H, \bZ(H))$. That is, we have $U= \Diag(U_1, U_2) $ where $U_2$ is a quantum isomorphism from the graph induced on $\bZ(G)$ to the graph induced on $\bZ(H)$. 
	\end{lemma}
	\begin{proof}
		To prove the result, one has to show that for all $x\in V(G)$ and $a\in V(H)$, if $u_{ax} \neq 0 $, then $x\in \bZ(G)$ if and only if $a\in \bZ(H)$. We start by proving the following claim: for every $x\in \bZ(G)$ and $a\in V(H)$ such that $u_{ax} \neq 0$, we have that $a \in \bZ(H)$. To prove the claim, we fix $x\in \bZ(G)$ and $a\in V(H)$ such that $u_{ax} \neq 0$.
		
		By~\cite[Corollary 3.3]{Meunier2023}, we have that $U=\Diag(V,W)$ where $V$ and $W$ are two magic unitaries and $W$ is indexed by $Z(H) \times Z(G)$. In particular, $W$ is a quantum isomorphism from $G[Z(G)]$ to $H[Z(H)]$. 
		
		Firstly, we consider the case where $ Z(G) = \{r\}$ and $r$ is a cut vertex of $G$. In this case, we have $\bZ(G) = Z(G) = \{r\} $, so $x=r$. 
		Since the graphs induced on $Z(G)$ and $Z(H)$ are quantum isomorphic, we have $\abs{Z(H)} = 1$. Letting $Z(H)=\{s\}$, we have $W=(u_{sr})$ with $u_{sr}=1$. Since $r=x$ and $U$ is a magic unitary, we have that $a=s \in Z(H) \subseteq \bZ(H)$, as desired.

		Secondly, let us consider the case  where $ Z(G) = \{r\}$ and $r$ is an internal vertex of $G$. In this case $\bZ(G)$ is the unique block containing $r$. Once again, since  $Z(G)$ and $Z(H)$ are quantum isomorphic, we have $\abs{Z(H)} = 1$, say $Z(H)=\{s\}$. This implies that $W=(u_{sr})$ with $u_{sr}=1$. Thus, by~\Cref{lem:qi_blocks}, we have that $s$ is an internal vertex of $H$, and there exists a unique block containing $s$ which is $\bZ(H)$. We have $ 0 \neq u_{ax} = u_{ax}u_{sr} $. Since $x$ and $r$ are in a common block, by~\Cref{lem:qi_blocks}, so are $s$ and $a$. Since $s$ is only in the block $\bZ(H)$, we have $a \in \bZ(H)$, as desired. 
		
		Finally, we consider the case where $\abs{Z(G)} > 1$. In this case, let $ y, z \in Z(G)$ with $y \neq z$. We have
		$$
		0 \neq u_{ax} = u_{ax} \left( \sum_{c \in V(H)} u_{cz} \right) = \sum_{c \in V(H)} u_{ax} u_{cz}.
		$$
		So, there exists $c \in V(H)$ such that
		\begin{equation} \label{eq:1}
			u_{ax} u_{cz} \neq 0. 
		\end{equation} 
		Moreover, we have 
		$$
		0 \neq u_{ax} u_{cz} = u_{ax} \left( \sum_{b \in V(H)} u_{by} \right) u_{cz} = \sum_{b \in V(H)} u_{ax} u_{by} u_{cz}.
		$$
		Thus, there exists $b \in V(H)$ such that 
		\begin{equation} \label{eq:2}
			u_{ax} u_{by} u_{cz} \neq 0. 
		\end{equation} 
		From~\Cref{eq:2}, we have $u_{by} u_{cz} \neq 0 $. Since $y $ and $z$ are distinct, we have $ b \neq c $ by orthogonality of the rows of $U$. Moreover, since $u_{by} \neq 0$ and $u_{cz} \neq 0$, since $U$ is diagonal by block, we have that $ b\in Z(H)$ and $c \in Z(H)$. So in particular, $b, c \in \bZ(H)$. Since by~\Cref{prop:properties_of_blocks}, the blocks of $H$ intersect in at most one vertex, we have that 
		\begin{itemize}
			\item[(i)] $\bZ(H)$ is the unique block of $H$ containing $b$ and $c$. 
		\end{itemize}
		Now, if $a = b$ or $a=c$, we have $a \in \bZ(H)$, as desired. So, from now on we assume that $a$ is distinct from both $b$ and $c$.
		Applying~\Cref{lem:qi_blocks} to~\Cref{eq:1}, since $ x$ and $z$ are in a common block $\bZ(G)$, we obtain that 
		\begin{itemize}
			\item[(ii)] there exists a block $C$ containing both $a$ and $c$.
		\end{itemize}
		Similarly, using~\Cref{eq:2}, we have $u_{ax}u_{by} \neq 0$, and since $x $ and $y $ are in a common block $\bZ(G)$, applying~\Cref{lem:qi_blocks}, we obtain that 
		\begin{itemize}
			\item[(iii)] there exists a block $B$ containing both $a$ and $b$.
		\end{itemize}
		If neither $B$ nor $C$ is equal to $\bZ(H)$, then we have $B \neq C$ as well, since 
		$$ B \cap \bZ(H) = \{b\} \neq \{c\} = C \cap \bZ(H).$$
		Take a path $P_{ab}$ entirely in $B$ from $a$ to $b$, a path $P_{bc}$ entirely in $\bZ(H)$ from $b$ to $c$, and a path $P_{ca}$ entirely in $C$ from $c $ to $a$. Then, $P_{ab}P_{bc}P_{ca}$ is a cycle in $G$ passing through three distinct blocks, a contradiction to item (3) of~\Cref{prop:properties_of_blocks}. So, either $B = \bZ(H)$ or $C=\bZ(H)$, which implies $a \in \bZ(H)$, as desired. 
		This completes the proof of the claim.
		
		Now, let $a \in \bZ(H)$ and $x \in V(G)$ such that $u_{ax} \neq 0$. Note that $U^*$ is a quantum isomorphism from $H$ to $G$ and that $[U^*]_{xa} = u_{ax} \neq 0$. By applying the precedent claim to $U^*$, we obtain $x \in \bZ(G)$. This completes the proof of the lemma.
	\end{proof}

	Before proving our main theorem, let us point out some possible generalisations of the previous result. The first two statements of \Cref{remark:generalisations} can be easily proved with an inductive argument using the operation $\Gamma$ -- similar to the ones of \Cref{lem:anchored_qiso_implies_blocks_iso} and \Cref{thm:degree_cut}. We do not include their proofs here both because we do not use these generalisations and for the sake of brevity.
	
	\begin{remark} \label{remark:generalisations}
		 Let $U$ be a quantum isomorphism from a graph $G$ to a graph $H$. 
		\begin{enumerate}
			\item The proof of \Cref{lem:center-s-block-is-preserved} does not use in a key way the fact that the fixed subsets  are the centers of the graphs. Indeed, a similar proof would give the following result. Let $S \subseteq V(G)$ and $T \subseteq V(H)$ be such that $u_{ax} = 0$ if $x \in S$ and $a \notin T$ or if $x \notin S$ and $a \in T$. If~$S$ is contained in a unique block $b$ of $G$, then $T$ is contained in a unique block $d$ of $H$ and moreover, $U = \Diag(U_1, U_2)$ is diagonal by block where $U_2$ is a quantum isomorphism from~$b$ to~$d$.
			
			\item Note that for any graph $F$, we have that $\bZ(F)$ is a vertex of the block tree $\Tt(F)$. \Cref{lem:center-s-block-is-preserved} shows that these vertices are fixed by $U$. One can furthermore show that $U$ preserves the distances to these vertices in the block trees. More precisely, if $u_{ax} \neq 0$ for $x \in V(T)$ and $a \in V(H)$, then we have $ d_{\Tt(G)}(\rho_G(x), \bZ(G)) = d_{\Tt(H)}(\rho_H(a), \bZ(H))$. 
						
			\item  Finally, we remark that~\Cref{lem:center-s-block-is-preserved,lem:anchor_is_the_same} together imply that $Z(G)$ consists of a cut vertex of $G$ if and only if $Z(H)$ consists of a cut vertex of $H$.
		\end{enumerate}
	\end{remark}
	
	Now we can prove the main theorem of this section.
	
	\begin{theorem}\label{thm:pegah_conjecture}
		If $G$ and $H$ are two quantum isomorphic connected graphs, then 
		\begin{enumerate}
			\item there exists a morphism $\alpha: \Tt(G) \rightarrow \Tt(H)$ such that
			\begin{itemize}
				\item $\alpha$ is an isomorphism between the block trees of $G$ and $H$ that preserves the colours,
				\item for every block $b$ of $G$, we have $b \qi \alpha(b)$,
			\end{itemize}
			\item there exists a graph morphism $\beta: \Bb(G) \rightarrow \Bb(H)$ such that 
			\begin{itemize}
				\item $\beta$ is an isomorphism between the block graphs of $G$ and $H$,
				\item for every block $b$ of $G$, we have $b \qi \beta(b)$.
			\end{itemize}
		\end{enumerate}
	\end{theorem}
	
	\begin{proof}
		By~\Cref{lem:center-s-block-is-preserved}, we have that $(G, \bZ(G))$ and $(H, \bZ(H))$ are quantum isomorphic connected anchored graphs. Thus the result follows from~\Cref{lem:anchored_qiso_implies_blocks_iso}.
	\end{proof}

	Next, we prove a local version of~\Cref{thm:pegah_conjecture}, namely that quantum isomorphisms preserve the blocks containing a given vertex.
	
	\begin{theorem}\label{thm:degree_cut}
		Let $ U $ be a quantum isomorphism from $ G $ to $ H $. Let  $ x \in V(G) $ and let $b_1, \dots, b_k$ the blocks of $G$ containing $x$. Let $ a \in V(H) $ and let $d_1, \dots, d_l$ the blocks of $H$ containing $a$. If $ u_{ax} \neq 0 $, then $k=l$, and up to reordering $G[b_i] \qi H[d_i]$ for every $ 1 \leq i \leq k$.
	\end{theorem}
	
	\begin{proof}
		By~\Cref{thm:partitioned_graphs,lem:connected_components_partition}, it is enough to prove the theorem for connected graphs.
		
		For a graph $F$ and a vertex $ v\in V(F)$, let $\BB_F(v)$ be the set of blocks of $F$ containing~$v$. 
		
		We first prove the following statement for anchored graphs: \emph{let $ U $ be a quantum isomorphism from a connected anchored graph $ (G, R) $ to an anchored graph $ (H,S) $ and let $ x \in V(G) $ and $ a \in V(H) $. 
			If $ u_{ax} \neq 0 $, then there is a one-to-one correspondence between $\BB_G(x)$ and $\BB_H(a)$ such that the graphs induced on the corresponding blocks are quantum isomorphic.}
		
		We prove the statement by induction on the number of blocks of $G$ (which, by~\Cref{thm:pegah_conjecture}, is equal to the number of blocks of $ H $).
		
		If $G=K_1$, the result is immediate. Thus, we may assume that $G$ has at least one block.	
		Also, if $G$ has exactly one block, then by~\Cref{coro:qi_2connected}, $H$ is 2-connected as well. 
		Therefore, in this case, $G$ and $H$ are the only blocks containing $x$ and $a$ and the result is immediate.
		
		Now let $U$, $(G,R)$, $(H,S)$, $x\in V(G)$, and $a\in V(H)$ be as in the statement above and assume that the graphs are not 2-connected. 
		Also, assume that the statement holds for any other pair of connected anchored graphs with strictly less blocks. 
		Set $(G',R') = \Gamma(G,R)$ and $(H',S')=\Gamma(H,S)$, and note that by~\Cref{lem:qi_for_anchored_graphs}, $W=\Gamma(U)$ is a quantum isomorphism from the former to the latter. 
		Let $(G_0,R_0)$ and $(H_0,S_0)$ be the connected components of $G$ and $H$ containing $x$ and $a$ respectively. 
		Note that the number of blocks of $G_0$ and $H_0$ is less than that of $G$. Moreover, by~\Cref{lem:qi_for_disconnected_anchored_graphs}, $W[H_0, G_0]$ is a quantum isomorphism from $(G_0, R_0)$ to $(H_0,S_0)$.
		
		If $x\notin R$ (thus $a\notin S$), then $\BB_{G_0}(x)= \BB_{G'}(x)= \BB_{G}(x)$ and $\BB_{H_0}(a)=\BB_{H'}(a)=\BB_H(a)$. 
		Thus, by applying the induction hypothesis to $(G_0,R_0)$ and $(H_0,S_0)$, we have the required result.
		So, for the rest of the proof, we assume that $x\in R$ (and consequently, $a \in S$). 
		
		First, if $R=\{r\}$ consists of a cut vertex of $G$ and $S=\{s\}$ consists of a cut vertex of $H$, then $x=r$ and $a=s$. 
		So, let $\BB_G(x) = \{ b_1, \dots, b_k\} $ and $\BB_H(a) = \{d_1, \dots, d_l \}$. By construction, there exists a natural one-to-one correspondence between $ b_i $'s and the anchors $R_1, \dots, R_k$ of the connected components of $(G',R')$ such that $G[b_i] = G'[R_i]$. 
		Similarly, there exists a natural one-to-one correspondence between $ d_i $'s and the anchors $ S_1, \dots, S_l$ of the connected components of $(H',S')$ such that $H[d_i] = H'[S_i]$.
		Since $(G',R')$ and $(H',S')$ are quantum isomorphic as anchored graphs, using~\Cref{lem:qi_for_disconnected_anchored_graphs}, the result follows.
		
		Second, assume that $R$ is a block of $G$ containing $x$ and $S$ is a block of $H$ containing $a$. 
		If $x$ is an internal vertex of $R$, then $a$ is also an internal vertex of $S$ by~\Cref{lem:qi_blocks}. So $R$ and $S$ are the only blocks containing $x$ and $a$ respectively. The result then follows from~\Cref{lem:anchor_is_the_same}.  
		If $x$ is a cut vertex of $G$, then by~\Cref{lem:qi_blocks}, $s$ is a cut vertex of $H$. In this case, we have 
		\[\BB_{G_0}(x) = \BB_{G'}(x) = \BB_G(x) \setminus \{R\} \text{ and } \BB_{H_0}(a) = \BB_{H'}(a) = \BB_H(a) \setminus \{S\}.\]
		By applying the induction hypothesis to $(G_0,R_0)$ and $(H_0,S_0)$ as above and by using~\Cref{lem:anchor_is_the_same}, we obtain the required result. 
		
		This completes the proof of the statement for anchored graphs.
		Now, by~\Cref{lem:center-s-block-is-preserved}, $U$ is a quantum isomorphism from $(G, \bZ(G))$ to $(H, \bZ(H))$ such that $u_{ax} \neq 0$. Hence the result follows from the statement proved above.
	\end{proof}
	
	Let us close this section with two corollaries of~\Cref{thm:pegah_conjecture}.

	\begin{corollary}\label{cor:numberofblocksandcutvertices}
		Two quantum isomorphic graphs have the same number of blocks and the same number of cut vertices. 
	\end{corollary}
	\begin{proof}
		By~\Cref{lem:connected_components_partition}, it is enough to prove the statement for connected components, which is in turn immediate from~\Cref{thm:pegah_conjecture}.
	\end{proof}
	
	\begin{corollary} \label{cor:minimalpair}
		Let $n$ be the smallest integer such that there exists a pair of graphs on $n$ vertices which are quantum isomorphic but not isomorphic. Let $(G,H)$ be a pair of quantum isomorphic graphs on $n$ vertices. Then one of the following holds:
		\begin{enumerate}
			\item $G$ and $H$ are isomorphic, 
			\item $G$ and $H$ are connected and for the isomorphisms $\alpha$ and $\beta$ from~\Cref{thm:pegah_conjecture} we have furthermore that $\alpha(b) = b = \beta(b)$ (i.e. these are classically isomorphic) for every block $b$ of $G$,
			\item $G$ and $H$ are 2-connected.
		\end{enumerate}
		Moreover, the statement holds for $(G^c, H^c)$ as well.
	\end{corollary}
	\begin{proof}
		Let us assume that $G$ and $H$ are not isomorphic. It follows from~\Cref{lem:connected_components_partition} that $G$ and~$H$ are connected because otherwise, their connected components would be two-by-two isomorphic, since they each have strictly less than $n$ vertices. 
		Furthermore, assume that $G$ and $H$ are not 2-connected. In this case, each block of $G$ and of $H$ has strictly less than $n$ vertices, hence the result follows from~\Cref{thm:pegah_conjecture}. 
	\end{proof}

\end{document}